\documentclass[11pt]{amsart}

\usepackage{amssymb}
\usepackage{amsmath}
\usepackage{amscd}
\usepackage{curves}
\usepackage{epsfig}
\usepackage{geometry}
\usepackage{graphicx}
\usepackage{pstricks}
\usepackage{pstricks-add}
\usepackage{pst-grad} 
\usepackage{pst-plot} 
\usepackage{verbatim}
\usepackage{latexsym,eucal,amsfonts,amssymb,amsmath,graphicx}

\newtheorem{theorem}{Theorem}[section]
\newtheorem{lemma}[theorem]{Lemma}
\newtheorem{prop}[theorem]{Proposition}

\newtheorem{definition}[theorem]{Definition}

\newtheorem{remark}[theorem]{Remark}

\numberwithin{equation}{section}

\def \mr {{\mathbb R}}

\def \mcs {{\mathcal S}}

\def \mcm {{\mathcal M}}
\def \mca {{\mathcal A}}
\def \mce {{\mathcal E}}

\def \mcf {{\mathcal F}}

\def \mcl {{\mathcal L}}
\def \mcr {{\mathcal R}}

\def \mcv {{\mathcal V}}

\def \mn {{\mathbb N}}

\def \la {\lambda}   
\def \La {\Lambda}

\def \lap {\Delta}

\newcommand{\supp}{\text{supp }}
\newcommand{\diag}{\textrm{Diag}}
\newcommand{\inj}{\textrm{inj}}
\newcommand{\Id}{\operatorname{Id}}

\def \im {\operatorname{Im}}
\def \re {\operatorname{Re}}

\def \intx {\mathring{X}}

\def \p {\partial}
\def \ff {\text{ff}}
\def \fot {\frac{1}{2}}
\def \fnt {\frac{n}{2}}
\def \fracnt {\fnt}
\def \fnf {\frac{n^2}{4}}
\def \bpf{\begin{proof}}
\def \epf{\end{proof}}
\def \beq{\begin{equation*}}
\def \eeq{\end{equation*}}
\def\bsp{\begin{split}}
\def \mbr {\mathbb{R}}

\def \intxx {\intx\times\intx}
\def \xox {X\times_0 X}
\def \xo {\xox}
\def \xx {X\times X}
\def \hpl {H_{p_L}}
\def \hpr {H_{p_R}}
\def \ha {\frac{1}{2}}
\def \mc {\mathbb C}
\def \intd {\diag^\circ}
\def \Xh {X^2_{0, \hbar}}
\def \mck{\mathcal{K}}
\def \soh {\frac{\sigma}{h}}
\def \mcep {\mce_0}

\def \beqq {\begin{equation}}
\def \eeqq {\end{equation}}
\def \hql {H_{Q_L}}
\def \hqr {H_{Q_R}}
\def \mbb{\mathbb{B}}

\def \det {\text{det}}
\def \ltx {L^2(X)}
\def \las {\La_{S}}
\def \laex {\bar{\La}_{S}}
\def \betas {\beta_{S}}
\def \intxo{\text{int}(\xo)}
\def \lazero {\La_0}

\begin{document}
\title[Resolvent and Radiation Fields]{Resolvent and Radiation Fields on Non-trapping Asymptotically Hyperbolic Manifolds}
\author{Yiran WANG}
\email{wang554@purdue.edu}
\address{Department of Mathematics, Purdue University, 150 North University Street, West Lafayette, Indiana, 47907, USA}
\date{\today}

\begin{abstract}
We construct a semi-classical parametrix for the Laplacian on non-trapping asymptotically hyperbolic manifolds, which generalizes the construction of Melrose, S\'a Barreto and Vasy. As applications, we obtain high energy resolvent estimates and resonance free strips. Also, we prove that the Friedlander radiation field decays exponentially for initial data  in proper spaces. 
\end{abstract}

\maketitle

\section{Introduction}
Consider an $n + 1$ dimensional compact manifold $X$ with boundary $\p X$. The interior of $X$ is denoted by $\intx$. Let $\rho \in C^\infty(X)$ be a boundary defining function of $\p X$, i.e. $\rho > 0 \text{ in } \intx, \rho = 0 \text{ on } \p X \text{ and } d\rho \neq 0 \text{ on }\p X.$ When equipped with a Riemannian metric $g$, the manifold $(\intx, g)$ is called conformally compact if $G = \rho^2 g$ is non-degenerate up to $\p X$. If in addition $|d\rho|_G = 1$ at $\p X$, $(\intx, g)$ is called asymptotically hyperbolic. In this case, the sectional curvature approaches $-1$ along any curve toward $\p X$, see Mazzeo \cite{Ma1, Ma2}. As shown in \cite{Gr} (see also \cite{JS1}), for asymptotically hyperbolic manifolds, there is a product decomposition near $\p X$. That is, there exists a boundary defining function $x$ such that a neighborhood $U$ of $\p X$ is diffeomorphic to $[0, \epsilon)_x \times \p X$, and
\begin{equation}\label{prod}
g = \frac{dx^2 + H(x, y, dy)}{x^2},
\end{equation}
where $H$ is a family of Riemannian metrics on $\p X$ parametrized by $x$. It is pointed out in \cite{JS1} that the metric $g$ determines $x$ and $H$ up to a positive factor $f\in C^\infty(X)$. More precisely, the metric $g$ determines a conformal structure on $\p X$. 

The Laplace-Beltrami operator $\lap_g$ on $(\intx, g)$ is   essentially self-adjoint   on $\ltx$. The spectra of $\lap_g$ consists of essential spectrum $[\frac{n^2}{4}, \infty)$ and finite many eigenvalues in $(0, \frac{n^2}{4})$, see \cite{Ma1,Ma2, MM}. It is convenient to shift the spectrum and consider $P = \lap_g - \frac{n^2}{4}$, so that the essential spectrum of $P$ is $[0, \infty)$. For $\la \in \mc, \im\la < 0$, we use $\la^2$ as the spectral parameter and denote the resolvent by  
\begin{equation}
R(\la) = (\lap_g - \frac{n^2}{4} - \la^2)^{-1}.
\end{equation} 
By the spectral theorem, $R(\la)$ are bounded operators on $\ltx$ for $\im\la < -n/2$. From the work of Mazzeo and Melrose \cite{MM}, $R(\la)$ has a meromorphic continuation from $\im \la \ll 0$ to $\mc\backslash \frac{i}{2}\mn$, $\mn = \{1, 2, \cdots\}$ as bounded operators on weighted $L^2$ spaces. 
The poles of the continuation are called resonances. Guillarmou showed in \cite{G1} that the resolvent has a meromorphic continuation to $\mc$ if and only if the metric $g$ is asymptotically even. Otherwise, the points $\frac{i}{2}\mn$ are generally essential singularities of $R(\la)$.

In this work, we assume that $(\intx, g)$ is non-trapping, i.e. there is no complete geodesic contained in any compact set of $\intx$. 
The main goal of this work is to generalize a semi-classical parametrix construction of Melrose, S\'a Barreto and Vasy \cite{MSV} to asymptotically hyperbolic manifolds, and use it to study  the resolvent and radiation fields. Let $h = 1/|\re\la|$ and $\sigma = 1 + i \im\la/|\re\la|$, we  transform the problem to a semi-classical one by $P - \la^2 = h^{-2} P(h, \sigma)$, where 
\beq
P(h, \sigma) = h^2(\lap_g - \fnf) -\sigma^2.
\eeq
Roughly speaking, we look for an operator $G(h, \sigma)$ such that 
\beq
P(h, \sigma)G(h, \sigma) = \Id + E(h, \sigma),
\eeq
for $h\in (0, 1),$ and $\sigma$ in a box around $1 + i0$ i.e.
\beq
   \sigma \in \Omega_\hbar \doteq [1 - \epsilon, 1 + \epsilon]\times i [-Ch, Ch],
\eeq 
with some  $\epsilon >0$ small and $C > 0$. Here the remainder $E$ is of order $h^\infty$ as $h\rightarrow 0$. We call such an operator $G(h, \sigma)$ a semi-classical parametrix for the Laplacian. 

In \cite{MSV}, the authors constructed such a parametrix for the hyperbolic space with metric perturbations supported sufficiently close to the infinity. To gain a rough idea of the structure of the parametrix, we proceed with a simple example where the resolvent kernel is explicit. Consider the Poincar\'e ball model of the three dimensional hyperbolic space $(\mbb^{3}, g_0)$, where
\beq
\mbb^{3} = \{z \in \mbr^{3} : |z| < 1\}, \ \ g_0 = \frac{4dz^2}{(1 - |z|^2)^2}.
\eeq
The manifold is complete and without conjugate points, hence the distance function $r_0(z, z')$ is smooth away from $z = z' \in \intx$. Let $\lap_0$ be the Laplace-Beltrami operator and $P_0 = \lap_0 - 1$. The Schwartz kernel (trivialized by the volume form) of the resolvent $R_0(\la)$ is 
$
R_0(\la, z, z') = \frac{1}{4\pi} e^{-i\la r_0(z, z')} \sinh^{-1} r_0(z, z'),
$
see \cite{MSV} and the references there. Now for the semi-classical operator $P_{0}(h, \sigma) = h^2(\lap_0 - 1) - \sigma^2$, the Schwartz kernel of its inverse $G_{0}(h, \sigma)$ is simply 
\beqq\label{semiker}
G_0(h, \sigma, z, z') = h^{-2}\frac{e^{-i\soh r_0(z, z')}}{4\pi \sinh r_0(z, z')}.
\eeqq
This kernel is singular at $r_0 = 0$ i.e. $z = z'$, and is a smooth oscillatory function away from $z = z'$ with phase function $-\sigma r_0$. It is important to notice that  the kernel has  complicated compound asymptotics as $h\rightarrow 0, z, z'\rightarrow \p \mbb^3$ and $z = z'$. The authors of \cite{MSV} were able to describe this clearly by working on a specially designed semi-classical blown-up space. 

In  general, one expects the semi-classical parametrix to exhibit similar structures as $G_0(h, \sigma)$. The Schwartz kernel of $G(h, \sigma)$ can be regarded as a distribution (trivialized by the volume form) on $\intx \times \intx \times [0, 1)$. We follow \cite{MSV}  and work on a blown-up space $X^2_\hbar$ defined in Section 3. The blown-up space is a compact manifold with corners and has five boundary faces $\mcl, \mca, \mcr, \mcs, \mcf$, see Figure \ref{fig4}. When lifted to $X^2_\hbar$, the kernel of $G(h, \sigma)$ can be decomposed into two parts. The first part only captures the conormal singularities at the diagonal of $X^2_\hbar$. Such operators are pseudo-differential operators and denoted by $\Psi_{0, \hbar}^m(X)$ with $m$ indicating the order. The novelty of our analysis lies in the second part, which captures the asymptotic behavior of the kernel at  the boundary faces of $X^2_\hbar$. The kernel belongs to a class of oscillatory functions associated to a Lagrangian submanifold and conormal to boundary faces of $X^2_\hbar$. Such operators are denoted by $I_\hbar^{ a + i\soh, \mu - \kappa, c + i\soh}(X, \La)$, see Section 3 for detail explanations. The two operator spaces are combined to $\Psi_{0, \hbar}^{m, a + i\soh, \mu-\kappa, c + i\soh}(X, \La)$. Our main result is 
\begin{theorem}\label{main0}
Assume $h\in [0, 1), \sigma \in \Omega_\hbar$. There exist $\kappa \geq 0$ and two operators $G(h, \sigma)$ and $E(h, \sigma)$,  whose Schwartz kernels are holomorphic in $\sigma$ and satisfy
\beq
G(h, \sigma)\in \Psi_{0,\hbar}^{-2, \frac{n}{2} + i\soh, -\frac{n}{2} - 1 - \kappa, \frac{n}{2}+ i\soh}(X, \La), \ \  \beta_\hbar^*E(h, \sigma) \in \rho_\mcs^\infty\rho_\mcf^\infty \rho_\mca^\infty\rho_\mcr^\infty\rho_\mcl^{\fnt + i\frac{\sigma}{h}}C^\infty(\Xh),
\eeq
such that 
\begin{equation*}
P(h, \sigma)G(h, \sigma) = \Id + E(h, \sigma).
\end{equation*}
\end{theorem}

We remark that the index $-\fnt - 1- \kappa$ can be thought as the vanishing order of $G(h, \sigma)$ in $h$, and $\kappa$ is related to the caustics of the Lagrangian. When the Lagrangian has no caustics as in the case of \cite{MSV}, $\kappa = 0$ and the result is the same as Theorem 5.1 of \cite{MSV}. However, in general, there is a loss in the vanishing order.  The Lagrangian central to our analysis can be described as following. Let $(z, \zeta)$ be local coordinates of $T^*\intx$. Consider the Hamiltonian function $p(z, \zeta) = \ha |\zeta|_{g^*(z)}^2$, where $g^*$ denotes the dual metric. Let $p_L$ be the lift of $p$ to $T^*\intx\times T^*\intx$ from the left and $\hpl$ be the Hamilton vector field of $p_L$. Denote
\beq
\Omega = \{(z, z', \zeta, \zeta') \in T^*(\intx\times\intx) : z = z', \zeta = -\zeta', p(z, \zeta) = \ha\}.
\eeq
We consider a non-conic Lagrangian submanifold defined as the flow out of $\Omega$ under $\hpl$ i.e.
\beqq\label{lade}
\Lambda = \bigcup_{t \geq 0} \exp t \hpl (\Omega).
\eeqq
Some equivalent descriptions are discussed in Section 2. In fact, if the distance function $r_0$ is smooth away from the diagonal, for example in the case of $(\mbb^3, g_0)$ or geodesic convex asymptotically hyperbolic manifolds, we have an explicit picture
\beq
\La \backslash \Omega = \{(z, z', d_z r_0(z, z'), -d_{z'} r_0(z, z')) : z, z' \in \intx, z\neq z'\},
\eeq
see \cite{MSV} and \cite{Pre}. In view of \eqref{semiker}, it is clear that that $G_0(h, \sigma, z, z')$ is an oscillatory function associated to $\La$ away from the diagonal. The general case is similar. To understand the asymptotic behavior of $G_0$ at $\p (X\times X)$, the point of view of \cite{MSV} is to study the smooth extension of $\La$ to the boundary  in some properly designed rescaled vector bundles. In \cite{Pre}, the smooth extension is addressed for general non-trapping asymptotically hyperbolic manifolds by using   different techniques. Our local parametrization of the Lagrangian follows from the results of \cite{Pre}, and we review them in Section 2. Also, our treatment bears some similarily with that of Duistermaat on the Lagrangian level, see Section 5.2 of \cite{D}.

As in \cite{MSV}, the high energy i.e. $|\la|\rightarrow \infty$ behavior of the resolvent is an immediate consequence. 
\begin{theorem}\label{main1}
Let $\rho$ be a boundary defining function of $\p X$. Assume $a, b > \im\la, \ \ a + b \geq 0$. For any $C_2 > 0$, there exists $C_1 >0$ such that the weighted resolvent $\rho^aR(\la)\rho^b$ has a holomorphic continuation to 
\beqq
\{\la\in\mc: |\re\la| > C_1, \ \ \im\la < C_2\}
\eeqq
as bounded operators on $\ltx$. 
Moreover, there exist $C>0$ and $\kappa \geq 0$ such that
\begin{equation}\label{main1est}
\|\rho^a R(\la) \rho^b f\|_{\ltx} \leq C |\la|^{\fracnt - 1 + \kappa}\|f\|_{\ltx}. 
\end{equation}
\end{theorem}

For  asymptotically hyperbolic manifolds (not necessarily non-trapping), Guillarmou \cite{G} proved there is no resonance in a region exponentially close to the real axis. For conformally compact manifold with constant negative curvature outside a compact set, he obtained resonance free regions as in Theorem \ref{main1}. For asymptotically hyperbolic manifold with even metric (in the sense of Guillarmou \cite{G1}), Vasy \cite{V} developed a whole microlocal machinery to obtain strips of holomorphic continuation as well as high energy resolvent estimates. We remark that Theorem \ref{main1} does not require the asymptotically hyperbolic metric to be even to any order. The resolvent estimates can be used to obtain, for example, the exponential decay of local energy for linear wave equations on such manifolds.

As another application, we study the radiation fields in Section 5. The radiation fields were introduced by Friedlander \cite{F, F1} for the wave equations on asymptotically Euclidean spaces. They are concrete realizations of the Lax-Phillips translation representation of the wave group, see \cite{LP}. In \cite{SaW}, S\'a Barreto and Wunsch showed that for non-trapping asymptotically Euclidean and hyperbolic manifolds, the radiation fields are Fourier integral operators associated to the sojourn relations.  

S\'a Barreto \cite{Sa} studied the radiation fields in the asymptotically hyperbolic setting. Consider the linear wave equation
\beq
\begin{split}
&(D_t^2 - \lap_g + \frac{n^2}{4})u(t, z) = 0 \text{ on } \mr_+\times \intx, \\
& u(0, z) = f_1(z), \ \  D_tu(0, z) = f_2(z), \ \ f_1, f_2 \in C^\infty_0(\intx).
\end{split}
\eeq
When the product decomposition \eqref{prod} is valid, we use $z = (x, y)$. By Theorem 2.1 of \cite{Sa},
the forward radiation field $\mcr_+: C^\infty_0(\intx)\times C^\infty_0(\intx) \rightarrow C^\infty(\mr\times \p X)$ is defined as
\begin{equation*}
\mcr_+(f_1, f_2)(s, y) = x^{-\fracnt}D_t u(s - \log x, x, y)\mid_{x = 0}.
\end{equation*}
The backward radiation field $\mcr_-$ is defined by reversing the time direction.

Here, we are interested in the decay of $\mcr_+$ as $s\rightarrow \pm\infty$. For odd dimensional Euclidean spaces, the radiation fields are given by derivatives of the Radon transform of the initial data $f_1, f_2$, see Section 4.2 of Lax and Phillips \cite{LP}. If $f_1, f_2$ are rapidly decaying Schwartz functions, the radiation fields are rapidly decaying in $s$ as $s\rightarrow \pm \infty$. Baskin, Vasy and Wunsch \cite{BVW} recently studied the radiation fields on  perturbations of the Minkowski space $M$. One of their results is that if the dimension of $M$ is even, the radiation field 
decays like $s^{-\infty}$ as $s\rightarrow \infty$. 

Roughly speaking, our Theorem \ref{main2} says that for $\epsilon$ small, and $f_1, f_2$ have sufficent regularity and decay at infinity,
\beq
\mcr_{+}(f_1, f_2)(s, y) = e^{-\epsilon |s|} \mcr_0(s, y), 
\eeq
where $\mcr_0 \in L^\infty(\mbr_s; L^2(\p X))$ and smooth in $s\in\mbr\backslash\{0\}$. In particular, the radiation field decays exponentially as $s\rightarrow \pm\infty.$ 
As pointed out in \cite{Sa}, the radiation field is related to the transposed Eisenstein  function $\mce(\la)$,  
 while $\mce(\la)$ can be defined as the boundary value of  $\rho^{-\fracnt - i\la}R(\la)$ at $\rho = 0$. In the non-semi-classical setting, Joshi and S\'a Barreto \cite{JS1} gave a clear description of the kernel structure of $\mce(\la)$. 
Here we study $\mce(\la)$ for large $|\re\la|$, and show that $\mce(\la)$ has a holomorphic extension to strips. 
These are carried out in Section 5.

\section{The Lagrangian and its parametrization}
\subsection{The Lagrangian  submanifold and its extension. }
We review the main construction of S\'a Barreto and Wang \cite{Pre}. Let $(\intx, g)$ be a $n+1$ dimensional non-trapping asymptotically hyperbolic manifold. The cotangent bundle $T^*\intx$ is a symplectic manifold with the canonical $2$-from $\omega$. 
In local coordinates $(z, \zeta)$ of $T^*\intx$, 
\beq
\omega = \sum_{j = 1}^{n+1} d\zeta_j\wedge dz_j.
\eeq
A submanifold $\La$ of $T^*(\intx)$ is called Lagrangian if $\text{dim} \La = \text{dim} \intx$ and $\omega$ vanishes on $\La$. Let $g^*$ be the induced metric on $T^*\intx$. In particular, $g^*_{ij} = (g^{-1})_{ij} = g^{ij}$. Consider a Hamiltonian function $p(z, \zeta) = \frac{1}{2}|\zeta|_{g^*}^2 = \fot \sum_{i, j = 1}^{n+1} g^{ij}\zeta_i \zeta_j.$ The Hamilton vector field $H_p$ is defined through $\omega(H_p, \cdot) = - dp$. In local coordinates, 
\begin{equation*}
H_p = \frac{\p p}{\p\zeta}\cdot\frac{\p}{\p z} - \frac{\p p}{\p z}\cdot\frac{\p}{\p \zeta}.
\end{equation*} 
The integral curves of $H_p$ are called bicharacteristics. Denote the unit sphere bundle of $T^*(\intx)$ by $S^*(\intx)$, then $S^* (\intx) = \{(z, \zeta)\in T^*\intx: p(z, \zeta) = \ha\}.$ Since $H_p p = 0$, $H_p$ is tangent to $S^*(\intx)$ and the bicharacteristics stay in $S^*(\intx)$. It is a well-known fact that the projection of bicharacteristics to $\intx$ are  geodesics of $(\intx, g)$, see  Section 2.C.8 of \cite{GFL}. 

Now consider the product manifold $\xx$. Let $\pi_L, \pi_R: T^*X\times T^*X\rightarrow T^*X$ be the projection to the left, right factor respectively. With the natural identification of $T^*X\times T^*X$ and $T^*(X\times X)$, we also use $\pi_L$, $\pi_R$ for the projections $T^*(X\times X) \rightarrow T^*X$. The cotangent bundle $T^*(X\times X)$ is a symplectic manifold with canonical $2$-form $\tilde\omega$. In local coordinate $(z, \zeta, z', \zeta')$ of $T^*(\xx)$,
\beq
\tilde\omega = \pi_L^*\omega + \pi_R^*\omega = \sum_{j = 1}^{n+1} d\zeta_j\wedge dz_j + \sum_{j = 1}^{n+1} d\zeta'_j\wedge dz'_j.
\eeq 

Set $p_L = \pi_L^* p, p_R = \pi_R^*p$. The Hamilton vector fields are denoted by $\hpl, \hpr$ respectively. Let 
\begin{equation}\label{sdiag}
\Omega = \{(z, \zeta, z', \zeta') \in T^*(\intx\times\intx) : z = z', \zeta' = -\zeta,  |\zeta|_{g^*} = 1\}.
\end{equation}
Consider the flow out of $\Omega$ under $\hpl$ or $\hpr$, i.e.
\beq
\La_L = \bigcup_{t \geq 0} \exp t \hpl (\Omega), \ \ \La_R =  \bigcup_{t\geq 0} \exp t \hpr (\Omega). 
\eeq
It follows from the definition that $\tilde\omega$ vanishes on $\Omega$, a $2n+1$ dimensional submanifold. The Hamilton vector fields $\hpl, \hpr$ are non-degenerate and the symplectic form $\tilde\omega$ is invariant under their flow. So the flow out is a $2n+2$ dimensional manifold on which $\tilde \omega$ vanishes. Hence $\La_L, \La_R$ are Lagrangian submanifolds of $T^*(\intx\times\intx)$. 

It turns out that $\La_L = \La_R$ because for any $(z, \zeta), (z', \zeta') \in S^*(\intx)$  and $(z', \zeta') = \exp t H_p(z, \zeta)$ with $t \geq 0$, one can reverse the direction of the flow to get $(z, -\zeta) = \exp t H_p (z', -\zeta').$ By the commutativity of $\hpl, \hpr$, we  arrive at $\La_L = \La_R$. Moreover, we  define 
\begin{equation}\label{defla}
\La \doteq \bigcup_{t_1, t_2 \geq 0} \exp t_2 \hpl \circ \exp t_1 \hpr (\Omega),
\end{equation}
which is the same manifold as $\La_L$ and $\La_R$. 

Next we consider the wave operators
\beq
\square_L = \ha( D_t^2 - \lap_{g(z)}),\ \ \square_R = \ha(D_t^2 - \lap_{g(z')}) .
\eeq
The symbols of these operators are 
\beq
Q_L(t, z, z'; \tau, \zeta, \zeta') = \ha(\tau^2 - |\zeta|_{g^*(z)}^2), \ \ Q_R(t, z, z'; \tau, \zeta, \zeta') = \ha(\tau^2 - |\zeta'|_{g^*(z')}^2).
\eeq
We consider the flow out of 
\beq
\Sigma \doteq \{(t, z, z'; \tau, \zeta, \zeta')\in T^*(\mbr\times \intxx) : t = 0, z = z', \zeta = -\zeta', |\zeta|^2_{g^*} = \tau^2 \neq 0\}
\eeq
under $H_{Q_L}$ and $H_{Q_R}$ i.e.
\begin{equation}\label{defla1}
\bsp
\tilde \La \doteq \bigcup_{t_1, t_2 \geq 0} \exp t_2 H_{Q_L} \circ \exp t_1 H_{Q_R} (\Sigma).
\end{split}
\end{equation}
We check that $\tilde\La$ is a conic Lagrangian submanifold of $T^*(\mr\times \intx\times \intx)$. Recall that a Lagrangian $\La_{c} \subset T^*X\backslash 0\doteq \{(z, \zeta) \in T^*X : \zeta \neq 0\}$ is conic if $(z, \zeta)\in \La_{c} \Longrightarrow (x, t\zeta)\in \La_{c}, \forall t > 0.$ By the commutativity of $H_{Q_L}$ and $H_{Q_R}$, we have 
\beq
\tilde\La = \bigcup_{t_2 \geq 0} \exp t_2 H_{Q_L} (\Sigma) = \bigcup_{t_1 \geq 0}  \exp t_1 H_{Q_R} (\Sigma).
\eeq
Notice that $\Sigma$ is a $2n+2$ dimensional conic submanifold, and the canonical two form $d\tau\wedge dt + \tilde \omega$ vanishes on $\Sigma$. The symbols $Q_L, Q_R$ are homogeneous and the Hamilton vector fields $H_{Q_L}, H_{Q_R}$ are non-degenerate. Therefore, $\tilde \La$ is a conic Lagrangian submanifold. It is important to observe that the projection of $\tilde \La\mid_{\{\tau = 1\}}$ to $T^*(\intx\times\intx)$ is just $\La$. 

To describe the smooth extension of $\tilde\La$ in \cite{Pre},  we  recall the $0$-blown-up space defined by Mazzeo and Melrose \cite{MM}, see also \cite{MSV}. Let $\diag = \{(z, z') : z = z' \in X\}$, and 
\beq
\p \diag = \{(z, z') \in \p X\times \p X : z = z'\} = \diag\cap(\p X\times \p X).
\eeq
The $0$-blown up space, denoted by $X\times_0 X$, is defined as 
\beq
X\times_0 X = (X\times X\backslash \p\diag) \sqcup S^*_{++}(\p\diag),
\eeq
where $S^*_{++}(\diag)$ denotes the doubly inward pointing spherical bundle of $T^*_{\p\diag} (X\times X)$. Let 
$\beta_0: X\times_0 X \rightarrow X\times X$ be the blow-down map. Then $X\times_0 X$ can be equipped with a topology and smooth structure such that $\beta_0$ is smooth. The blown-up space has 3 boundary hypersurfaces. The left and right faces, denoted by $L$ and $R$, are defined as the closure of $\beta_0^{-1}(\p X\times \intx)$ and $\beta_0^{-1}(\intx\times \p X)$ respectively. The front face $\text{ff}$ is the closure of $\beta_0^{-1}(\p \diag)$. We denote the lifted diagonal  by $\diag_0$, which is  the closure of $\beta_0^{-1}(\diag\cap (\intx\times \intx))$.
See Figure \ref{fig1}.

\begin{figure}[htbp]
\centering
\scalebox{0.7}
{
\begin{pspicture}(0,-3.8529167)(17.765833,3.8729167)
\rput(3.7458334,-0.14708334){\psaxes[linewidth=0.04,labels=none,ticks=none,ticksize=0.10583333cm](0,0)(0,0)(4,4)}
\rput(0.74583334,-3.1470833){\psaxes[linewidth=0.04,labels=none,ticks=none,ticksize=0.10583333cm](0,0)(0,0)(4,4)}
\psline[linewidth=0.04cm,arrowsize=0.05291667cm 2.0,arrowlength=1.4,arrowinset=0.4]{->}(3.7458334,-0.14708334)(0.24583334,-3.6470833)
\psline[linewidth=0.04cm,dotsize=0.07055555cm 3.0]{*-}(2.2458334,-1.6470833)(4.7458334,3.1729167)
\psline[linewidth=0.04cm](0.74583334,0.85291666)(3.7458334,3.8529167)
\psline[linewidth=0.04cm](4.7458334,-3.1470833)(7.7458334,-0.14708334)
\rput(2.2458334,-1.6470833){\psaxes[linewidth=0.04,arrowsize=0.05291667cm 2.0,arrowlength=1.4,arrowinset=0.4,labels=none,ticks=none,ticksize=0.10583333cm]{->}(0,0)(0,0)(5,5)}
\usefont{T1}{ptm}{m}{n}
\rput(5.295833,2.6979167){$\diag$}
\usefont{T1}{ptm}{m}{n}
\rput(2.6,-2){$\p\diag$}
\usefont{T1}{ptm}{m}{n}
\rput(7.4058332,-1.8420833){$x'$}
\usefont{T1}{ptm}{m}{n}
\rput(2.0358334,3.3779166){$x$}
\usefont{T1}{ptm}{m}{n}
\rput(1.35,0.5179167){$X$}
\usefont{T1}{ptm}{m}{n}
\rput(4.1658335,-2.6820834){$X$}
\usefont{T1}{ptm}{m}{n}
\rput(0.97583336,-3.6020834){$y - y'$}
\rput(13.745833,-0.14708334){\psaxes[linewidth=0.04,labels=none,ticks=none](0,0)(0,0)(4,4)}
\rput(10.745833,-3.1470833){\psaxes[linewidth=0.04,labels=none,ticks=none](0,0)(0,0)(4,4)}
\psline[linewidth=0.04cm](13.745833,3.8529167)(10.745833,0.85291666)
\psline[linewidth=0.04cm](17.745832,-0.14708334)(14.745833,-3.1470833)
\psline[linewidth=0.04cm](11.865833,-2.0270834)(10.745833,-3.1470833)
\psarc[linewidth=0.04](13.425834,-2.0270834){1.56}{90.0}{180.0}
\psarc[linewidth=0.04](11.885834,-0.50708336){1.54}{268.5312}{1.27303}
\psline[linewidth=0.04cm](13.745833,-0.14708334)(13.425834,-0.48708335)
\psarc[linewidth=0.04](12.155833,-1.9770833){1.03}{28.855661}{82.11686}
\psline[linewidth=0.04cm,dotsize=0.07055555cm 3.0]{*-}(12.765833,-1.1470833)(14.725833,3.1129167)
\usefont{T1}{ptm}{m}{n}
\rput(15.245833,2.6579165){$\diag_0$}
\usefont{T1}{ptm}{m}{n}
\rput(14.445833,-1.9970833){\large $L$}
\usefont{T1}{ptm}{m}{n}
\rput(11.635834,0.8229167){\large $R$}
\usefont{T1}{ptm}{m}{n}
\rput(12.325833,-1.6020833){$\ff$}
\psline[linewidth=0.04cm,arrowsize=0.05291667cm 2.26,arrowlength=1.4,arrowinset=0.4]{->}(10.1258335,3.4129167)(8.105833,3.4129167)
\usefont{T1}{ptm}{m}{n}
\rput(9.155833,2.9029167){\large $\beta_0$}
\end{pspicture}
}
\caption{The $0$-blown-up space $X\times_0X$.}
\label{fig1}
\end{figure}
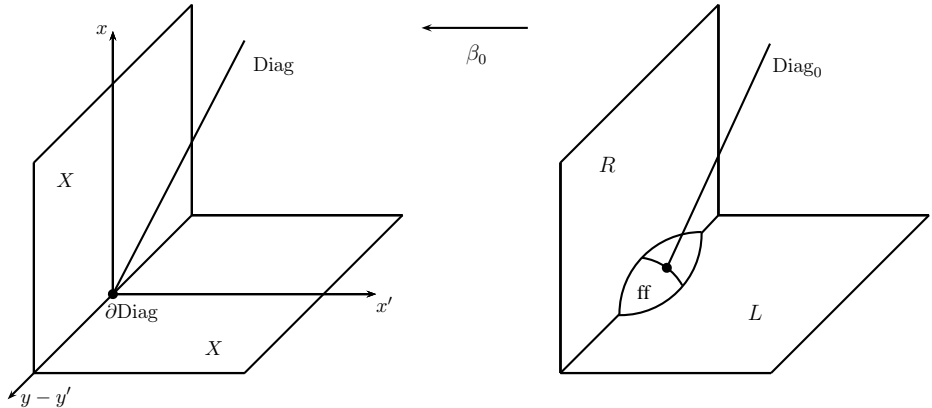

Since $\beta_0$ is a diffeomorphism when restricted to the interior of $\xo$ denoted by $\intxo$, it induces a symplecmorphism from $T^*(\intxo)$ to $T^*(\intx\times\intx)$, which we denote by $\beta_0^*$. We also use $\beta_0$ for $\mbr\times \xo \rightarrow \mbr\times X\times X$. So $\beta_0^*\tilde\La$ is a conic Lagrangian of $T^*(\mbr_t\times \xo)$. More explicitly, it can be expressed as
\beq
\beta_0^*\tilde \La = \bigcup_{t_1, t_2\geq 0} \exp t_2 \beta_0^*\hql \circ \exp t_1 \beta_0^*\hqr (\beta_0^*\Sigma).
\eeq
Here $\beta_0^*$ is also used for the lift of vector fields, i.e. $\beta_0^*\hqr = (\beta_0^{-1})_* \hqr$ in a conventional notation. 

Let $\rho_L, \rho_R$ be  boundary defining functions for the left, right faces respectively. Consider the following singular change of variables,
\beq
\bsp
\mcm : \mbr_t\times(\xo)\backslash (L\cup R) &\rightarrow \mbr_s \times\xo,\\
(t, m) &\rightarrow (s, m) \doteq (t + \log\rho_L(m) + \log \rho_R(m), m) .
\end{split}
\eeq
Let $\betas = \beta_0\circ\mcm$. The main result Theorem 1.1 of \cite{Pre} is
\begin{theorem}\label{extension}
Let $(\intx, g)$ be a non-trapping asymptotically hyperbolic manifold. The manifold $\las \doteq \betas^*\tilde\La$ can be smoothly extended up to the boundary of $T^*(\mbr_s\times\xo)$. 
\end{theorem}

The idea of the proof is to observe that
\begin{gather*}
q_L = \frac{1}{\rho_L}\betas^*Q_L, \ \ q_R = \frac{1}{\rho_R}\betas^*Q_R
\end{gather*}
are smooth functions on $T^*(\mbr_s\times \xo)$, i.e. smooth up to the boundary. From the basic property of Hamilton vector fields $H_{fg} = fH_g + g H_f$ for any $f, g$ smooth, we conclude that $H_{q_L} = \frac{1}{\rho_L} \betas^*H_{Q_L}, H_{q_R} = \frac{1}{\rho_R} \betas^*H_{Q_R}$ on $\las$, because the symbols vanish there.  So the integral curves of $H_{q_L}, H_{q_R}$ are the same as those of $\betas^*H_{Q_L}, \betas^*H_{Q_R}$ but with different parametrizations. Therefore, we can write
\begin{gather*}
\las = \bigcup_{t_1, t_2 \geq 0} \exp t_1  H_{q_L} \circ \exp t_2  H_{q_R} (\betas^*\Sigma).
\end{gather*}
The initial set $\betas^*\Sigma$ can be extended to a smooth submanifold of $T^*(\mbr_s\times \xo)$ up to the front face. A key step Lemma 2.1 of \cite{Pre}  shows that the Hamilton vector fields $H_{q_L}$ and  $H_{q_R}$ are 
transversal to $\mbr_s\times L$ and $\mbr_s\times R$ respectively. With the non-trapping assumption, all the integral curves arrive at $\mbr_s \times \p(\xo)$ in a finite time. Hence $\las$ has a smooth extension, denoted by $\laex$, to the boundary of $T^*(\mbr_s\times \xo)$.  Again by the non-trapping assumption, $\laex$ is indeed an embedded submanifold, see e.g. Section 5.1 of Duistermaat \cite{D}. We remark that the proof of Theorem \ref{extension} is based on a modification of an idea from Melrose, S\'a Barreto and Vasy \cite{MSV}. 

Actually, $\las$ can be extended across the boundary as a Lagrangian in the following sense. Consider a collar neighborhood of $\p X$ as in \eqref{prod}. This requires a choice of the boundary defining function $x$. For $\tilde\epsilon > 0$, we can extend the collar neighborhood to $(-\tilde\epsilon, \epsilon)_x\times \p X$ which gives an extension $\tilde X$ of $X$ across the boundary. Then $\tilde X\times \tilde X$ extends $X\times X$ across the boundary and corner, and  the blow-up gives an extension of $\xo$, denoted by $\widetilde\xo$. In the proof of Theorem 1.1 of \cite{Pre}, we can extend the functions $q_L, q_R$ smoothly to $\widetilde\xo$, therefore, the flow out $\las$ is extended across the boundary as a Lagrangian submanifold. We remark that such extensions of $\las$ are not necessarily unique.

Let $\sigma$ be the dual variable of $s$. One can check that after the change of variable $\mcm$, $\sigma$ is equal to the dual variable $\tau$ of $t$. 

\begin{lemma}\label{lacpt}
$\laex|_{\{\sigma = 1\}}$ is a compact  submanifold of $\laex$.
\end{lemma}
\bpf
The Hamiltonian functions $Q_L, Q_R$ do not depend on $t$, thus $q_L, q_R$ do not depend on $s$. Therefore the fiber $\sigma$ is constant along the Hamiltonian flow, and $\laex\mid_{\{\sigma = 1\}}$ is the flow out of $\betas^*\Sigma\mid_{\{\sigma = 1\}}$.  By the non-trapping assumption, all the integral curves reach $\mbr_s\times L$ and $\mbr_s\times R$, and they arrive at a finite time because the vector fields are transversal to the corresponding boundary faces. Therefore, it suffices to show that $\betas^*\Sigma\mid_{\{\sigma = 1\}}$ is compact. We first observe that the base $\{(s, m) : s = \log\rho_L(m) + \log \rho_R(m), m \in \diag_0\}$ is a compact submanifold  of $\mbr_s\times \xo$ because $\diag_0$ is compact and $\rho_L, \rho_R$ are smooth. The fibers are compact for $m$ in a compact set of the interior of $\diag_0$ since the metric is non-degenerate there. So the point is to verify that the compactness holds up to the front face. 
 
Near $\p X$ where the product decomposition \eqref{prod} is valid, we take $z = (x, y)$ as the local coordinates and $\zeta = (\xi, \eta)$ as the dual variables. Let $(x', y', \xi', \eta')$ be the local coordinate of the other copy. Then 
\beq
\Sigma\mid_{\{\tau = 1\}} = \{(0, x, y, x, y; 1, \xi, \eta, -\xi, -\eta) : x^2 \xi^2 + x^2 h(x, y, \eta) = 1\}.
\eeq 
For the $0$-blow up, we can use the following projective coordinate valid near $\ff$ and $L$,
\beq
x', \ \ X = x/x', \ \ Y = (y - y')/x', \ \ y'.
\eeq
The induced symplectic change of variables is
\begin{gather*}
\beta_0^* : (x, y, x', y'; \xi, \eta, \xi', \eta') \rightarrow (X, Y, x', y'; \la, \mu, \la', \mu')\\
\text{ where } \la = x' \xi, \ \ \mu = x'\eta, \ \ \la' = \xi' + \xi X + \eta Y, \ \ \mu' = \eta + \eta'.
\end{gather*}
Then 
\beq
\beta_0^*\Sigma|_{\{\tau = 1\}} = \{(0, X, Y, x', y'; 1, \la, \mu, \la', \mu') : X = 1, Y = 0, \la' = \mu' = 0, \la^2 + h(x', y', \mu) = 1\}.
\eeq
In the projective coordinates, $X$ is a boundary defining function of the left face. As discussed in \cite{Pre}, we can take the singular change of variable $\mcm$ to be $s = t + \log X$. This is because we are away from the right face hence $\rho_R$ is smooth, and changing boundary defining functions only results in a diffeomorphism of $T^*(\mbr_s\times \xo)$. Thus
\begin{gather*}
\mcm^* :  (t, X, Y, x', y';\tau, \la, \mu, \la', \mu') \rightarrow (s, X, Y, x', y'; \sigma, \tilde\la, \mu, \la', \mu')\\
\text{ where } \sigma = \tau, \ \ \tilde \la = \la - \tau/X.
\end{gather*}
Therefore, 
\beq
\betas^*\Sigma|_{\{\sigma = 1\}} = \{(0, 1, 0, x', y'; 1, \tilde\la, \mu, 0, 0) : (\tilde \la + 1)^2 + h(x', y', \mu) = 1\}.
\eeq 
Since $h$ is non-degenerate, this shows that $\betas^*\Sigma\mid_{\{\sigma = 1\}}$ is compact up to $\ff$. 
\epf

\subsection{Parametrization of the Lagrangian.} 
A conic Lagrangian $\La_c \subset T^*X$ can be locally parametrized by non-degenerate homogeneous phase functions, see e.g. \cite{D}. So for any $(z_0, \zeta_0) \in \La_c$, there exists a conic neighborhood $\Gamma \subset T^*X$ such that 
\beq
\La_c \cap \Gamma = \La_\phi \doteq \{(z, d_z\phi) : d_\theta \phi = 0\}.
\eeq
The phase function $\phi$ is smooth on some conic set of $X\times \mbr^{N_0}$, homogenous of degree one in $\theta$, and non-degenerate in a sense that $d_{z, \theta}d_z \phi$ has rank $N_0$. The following result describes the parametrization of $\lazero = \beta_0^*\La$, which is a Lagrangian submanifold of $T^*(\intxo)$.

\begin{prop}\label{lapara}
There exists finitely many phase functions $\phi_k \in C^\infty(U_k\times \Theta_k)$ where $U_k$ form an open covering of $\intxo$,  and $\Theta_k$ are bounded subsets of  $\mbr^{N_k}, N_k \in \mn$, such that 
\begin{gather*}
\lazero = \bigcup_{k = 1}^{M_\phi} \La_{k},\ \ \La_k = \{(m, d_m \phi_k) : m\in U_k,  d_\theta \phi_k = 0\}.
\end{gather*}
Moreover, the phase functions satisfy
\beq
\phi_k(m, \theta) = -\log \rho_L(m) - \log \rho_R(m) + F_k(m, \theta), 
\eeq
where $F_k$ have smooth extensions to the boundary of $\xo$. 
\end{prop}

\bpf
We start from a parametrization of  $\laex$. We know that $\laex$ is a conic Lagrangian of $T^*(\mbr_s\times\xo)$. Moreover, we can regard $\laex$ as a Lagrangian on some extension of $T^*(\mbr_s\times\xo)$ as discussed after Theorem \ref{extension}. Hence for any $(s_0, m_0; \sigma_0, \nu_0)\in\laex$, there exists a    neighborhood $I\times U$ of $(s_0, m_0)$ and a homogeneous non-degenerate phase function $\psi(s, m, \theta)$ on some conic set $I\times U \times \Theta \subset \mbr_s\times (\xo) \times \mbr^{N}$ such that 
\beq\label{eqlas}
\laex \cap T_{I\times U}^*(\mbr_s\times \xo) = \La_\psi = \{(s, m, d_s\psi, d_m \psi) : (s, m)\in I\times U, \ \ d_\theta\psi = 0\}.
\eeq
Because the phase function $\psi$ is non-degenerate, the map
\beq
i_\psi : \{(s, m, \theta) \in I \times U\times \Theta : d_\theta \psi = 0\} \rightarrow \laex \cap T_{I\times U}^*(\mbr_s\times \xo).
\eeq
is a diffeomorphism.  For any precompact neighborhood $W \subset \laex$ of $(s_0, m_0, 1, \nu_0) \in \laex\mid_{\{\sigma = 1\}}$, the preimage $i_\psi^{-1}(W)$ is precompact. By the compactness in Lemma \ref{lacpt},  $\laex\mid_{\{\sigma = 1\}}$ can be covered by a finite number of such precompact neighborhood $W_k$. Therefore, we find finitely many phase functions $\psi_k$ defined on  $I_k\times U_k \times \Theta_k$, such that 
\beq
\laex\mid_{\{\sigma = 1\}} = \bigcup_{k = 1}^{M_\phi} \La_{\psi_k}\mid_{\{\sigma = 1\}}. 
\eeq 
Here $I_k, \Theta_k$ are bounded sets because $\psi_k$ can be restricted to  precompact sets $i_{\psi_k}^{-1}(W_k)$.

Next, we get a parametrization of $\lazero$. We know that locally
\beq
\laex\mid_{\{\sigma = 1\}} = \{(s, m; 1, d_m\psi_k) : d_s\psi_k = 1, d_\theta \psi_k = 0\}. 
\eeq
By inverting the singular change of variable $\mcm$, we know for any $m\in \intxo$ that  
\begin{gather*}
(s, m; 1, \nu) \in \laex \text{ if and only if } (t, m; 1, \nu -\frac{d_m \rho_L}{\rho_L} - \frac{d_m \rho_R}{\rho_R} ) \in \beta_0^*\tilde \La.
\end{gather*}
By projecting to $T^*(\intxo)$, this is equivalent to $(m; \nu -\frac{d_m \rho_L}{\rho_L} - \frac{d_m \rho_R}{\rho_R} ) \in \lazero$. Now we define 
\beq
\phi_k(m, s, \theta) = \psi_k(s, m, \theta) - s - \log \rho_L(m) -\log \rho_R(m).
\eeq
Here we regard $s$ as a parameter in $\phi_k$. It is clear that
\begin{gather*}
d_s\phi_k = d_s\psi_k - 1, \ \ d_\theta\phi_k = d_\theta\psi_k,\\
d_m \phi_k = d_m\psi_k - \frac{d_m \rho_L}{\rho_L} - \frac{d_m \rho_R}{\rho_R}.
\end{gather*}
So locally, $\phi_k$ is a phase function parametrizing $\lazero$, i.e.
\beq
\lazero \cap T^*_{U_k}(\text{int}(\xo)) = \{(m, d_m\phi_k) : m\in U_k, \ \ d_s\phi_k = 0, d_\theta\phi_k = 0. \}
\eeq
Since $U_k$ covers $\intxo$, the proof is finished by taking $\Theta_k$ as $I_k\times \Theta_k$ and $F_k$ as $\psi_k$. 
\epf

\begin{remark}\label{rmkext}
In the above proof, we can get another covering of $\La_0$ by restricting $\phi_k$ to a smaller parameter set. First of all, we can shrink the precompact set $W_k$ so that they still cover $\laex|_{\{\sigma = 1\}}$ and $i_{\psi_k}^{-1}(W_k) =  \tilde I_k\times U_k\times \tilde \Theta_k$, where $\tilde I_k, \tilde \Theta_k$ are precompact in $I_k, \Theta_k$. Denote the restriction of $\psi_k$ by $\tilde \psi_k$. Then $\La_{\tilde \psi_k} \subset \La_{\psi_k}$ and $\laex|_{\{\sigma = 1\}}$ is covered by $\La_{\tilde \psi_k}$. By continuing the proof, we obtain the restriction $\tilde \phi_k$ of $\phi_k$ to $U_k\times (\tilde I_k\times \tilde \Theta_k)$ so that $\La_{\tilde\phi_k} \subset \La_{\phi_k}$ and $\La_0$ is covered by $\La_{\tilde \phi_k}$. This remark is useful in the proof of Lemma \ref{geo}.
\end{remark}

\section{Operators on Semi-classical Spaces}

The semi-classical blown-up space $X^2_{0, \hbar}$ can be constructed in a quite similar way as in Section 3 of \cite{MSV}.  On $X\times_0 X\times [0, 1)$, the set $\diag_0\times[0,1)$ intersect $X\times_0 X\times \{0\}$ transversally. We blow up this intersection to obtain $X_{0, \hbar}^2$, and let $\beta_{0, \hbar}: X_{0, \hbar}^2 \rightarrow X\times_0 X\times [0, 1)$
 be the blow down map. The final blow down map is $\beta_\hbar = \beta_{0, \hbar}\circ\beta_0 : X_{0, \hbar}^2 \rightarrow X\times X\times [0, 1)$. As a compact manifold with corners, $X_{0, \hbar}^2$ has five boundary faces, see Figure \ref{fig4}. The left, right faces, denoted by $\mcl, \mcr$, are the closure of $\beta_{0,\hbar}^{-1} (L \times[0, 1) ), \beta_{0, \hbar}^{-1}(R \times[0, 1))$ respectively.  The front face $\mcf$ is the closure of $\beta_{0, \hbar}^{-1}(\ff\times [0, 1) \backslash (\p \diag_0 \times \{0\})).$ The semiclassical front face $\mcs$ is the closure of $\beta_{0, \hbar}^{-1}(\diag_0\times\{0\})$. Finally, the semiclassical face $\mca$ is the closure of $\beta_{0,\hbar}^{-1}( (X\times_0 X\backslash \diag_0)\times \{0\})$. The lifted diagonal denoted by $\diag_\hbar$ is the closure of  $\beta_{0, \hbar}^{-1}(\diag_0\times(0, 1))$. See Figure \ref{fig4}.
\begin{figure}[htbp]
\centering
\scalebox{0.6} 
{\begin{pspicture}(0,-4.58)(19.76,5.0)
\psline[linewidth=0.04cm](5.58,0.0)(9.74,0.0)
\psline[linewidth=0.04cm](1.58,-4.0)(5.98,-3.98)
\psarc[linewidth=0.04](4.58,0.0){1.0}{0.0}{90.0}
\psarc[linewidth=0.04](0.58,-4.0){1.0}{0.0}{90.0}
\psline[linewidth=0.04cm,linestyle=dashed,dash=0.16cm 0.16cm, arrowsize=0.05291667cm 2.0,arrowlength=1.4,arrowinset=0.4]{->}(0.78,-3.8)(-0.42, -5)
\psline[linewidth=0.04cm](4.58,1)(4.58,5)
\psline[linewidth=0.04cm](0.58,-3)(0.58,1)
\psline[linewidth=0.04cm](0.58,-3)(4.58,1)
\psline[linewidth=0.04cm](1.58,-4)(5.58,0)
\pspolygon[linewidth=0.04](1.29,-3.29)(5.29,0.71)(6.4,3.8)(2.48,0.0)
\usefont{T1}{ptm}{m}{n}
\rput(7.76,3.505){\large $\text{Diag}_0\times[0, 1)_\hbar$}
\usefont{T1}{ptm}{m}{n}
\rput(0.7,-4.4){$h$}
\psline[linewidth=0.04cm](15.58,0.0)(19.74,0.0)
\psline[linewidth=0.04cm](11.58,-4.0)(15.98,-3.98)
\psarc[linewidth=0.04](14.58,0.0){1.0}{63.434948}{90.0}
\psarc[linewidth=0.04](10.58,-4.0){1.0}{0.0}{90.0}
\psline[linewidth=0.04cm](10.58,-3)(10.58,1)
\psline[linewidth=0.04cm](14.58,1)(14.58,5)
\psline[linewidth=0.04cm](10.58,-3)(14.58,1)
\psline[linewidth=0.04cm](11.58,-4)(15.58,0)
\pspolygon[linewidth=0.04](11.29,-3.29)(15.1,0.48)(16.17,3.15)(12.48,-0.2)
\usefont{T1}{ptm}{m}{n}
\rput(11.11,3.585){\large $\beta_{0,\hbar}$}
\usefont{T1}{ptm}{m}{n}
\rput(11.8,-0.165){\large $\text{Diag}_\hbar$}
\usefont{T1}{ptm}{m}{n}
\rput(15.99,-2.335){\large $\mcl$}
\usefont{T1}{ptm}{m}{n}
\rput(12.45,1.665){\large $\mcr$}
\usefont{T1}{ptm}{m}{n}
\rput(12.48,-2.755){\large $\mcf$}
\psline[linewidth=0.04cm,linestyle=dashed,dash=0.16cm 0.16cm](15.02,0.88)(16.0,3.58)
\psline[linewidth=0.04cm](15.46,0.5)(16.56,3.28)
\psarc[linewidth=0.04](15.24,0.72){0.3}{146.30994}{324.4623}
\psarc[linewidth=0.04](16.27,3.43){0.31}{145.30484}{347.4712}
\psarc[linewidth=0.04](14.58,0.0){1.0}{0.0}{29.47589}
\psline[linewidth=0.04cm,arrowsize=0.05291667cm 2.26,arrowlength=1.4,arrowinset=0.4]{->}(12.1258335,4)(10.105833,4)
\usefont{T1}{ptm}{m}{n}
\rput(15.72,1.545){\large $\mcs$}
\usefont{T1}{ptm}{m}{n}
\rput(18.25,1.385){\large $\mca$}
\end{pspicture} }
\caption{The semiclassical blown-up space $X^2_{0, \hbar}$ obtained from $X\times_0 X\times [0,1)$.}
\label{fig4}
\end{figure}
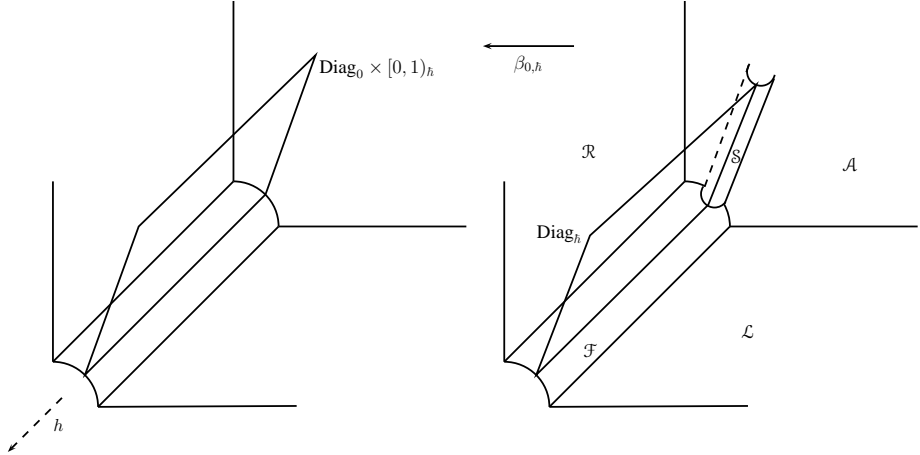

The volume form $dg$ of $(\intx, g)$ can be used to trivialize the distribution densities and give a well-defined distribution space $C^{-\infty}(\intx)$. In local coordinates, $dg = \sqrt{\det g} dz$. Let $P$ be an operator from $C_0^\infty(\intx)$ to $C^{-\infty}(\intx)$ with Schwartz kernel $K_P(z, z', h)$, then
\begin{equation*}
Pf(z, h) = \int K_P(z, z', h) f(z') dg(z'), \ \ f\in C_0^\infty(\intx).
\end{equation*}
In the rest of the paper, we often do not distinguish the notation of an operator and its Schwartz kernel. It should be clear from the context which one is referring to.  

By definition, the pseudo-differential operator space $\Psi^m_{0,\hbar}(X)$ consists of operators $P : C_0^\infty(\intx) \rightarrow C^{-\infty}(\intx)$ such that $\beta_\hbar^*K_P(z, z', h)$ is a conormal distribution of order $m$ to  $\diag_\hbar$, vanishing to infinite order at all faces,  except on $\mcf$ where it is smooth with distribution values, and on $\mcs$ where $\rho_\mcs^{n+1}\beta_\hbar^*(K_P)$ is smooth with distribution values. 

Next we define a class of oscillatory functions associated to $\lazero$, which is motivated by Definition 5.2.1 of Duistermaat \cite{D}. To fit in our context, we change the large parameter $\tau$ in the original definition by $1/h$.
\begin{definition}\label{oscfun}
Let $\Lambda$ be an immersed $C^\infty$ Lagrangian submanifold in $T^*X$. An oscillatory function $u(z, h)$ associated to $\La$ of order $\mu$ is a locally finite sum of integrals of the form
\begin{equation*}
I(z, h) = \int e^{-i\frac{\phi(z, \theta)}{h}} b(z, \theta, h) d\theta, \ \ \theta\in \mr^N, 
\end{equation*}
where $\Lambda_\phi = \{(z, d_z\phi) \in T^*X: d_\theta\phi = 0\}$ is a piece of $\Lambda$, $d_{x, \theta}d_\theta \phi$ has rank $N$ at  $d_\theta\phi = 0$, and
\begin{gather*}
b(z, \theta, h)\sim \sum_{j = 0}^{\infty} b_j(z, \theta) h^{\mu - \frac{1}{2}N + j},
\end{gather*}
where $b(z, \theta, h)$ vanishes for $\theta$ outside a fixed compact set of $\mr^N$.
\end{definition}
As discussed at the end of Section 5.2 of \cite{D}, the oscillatory integral $I(z, h) = O(h^{\mu - \ha \kappa_0})$ locally uniformly near $z_0$ as $h \rightarrow 0$, where $\kappa_0$ is the dimension of the intersection of the tangent space $T_{(z_0, \zeta_0)} (\La_\phi)$ and the fiber of the cotangent bundle $T_{(z_0, \zeta_0)}(\text{fiber})$ with $\zeta_0 = d_z\phi(z_0, \zeta_0)$. In particular,  $\kappa_0 \leq \text{dim}(X)$, and the points of $\La$ where $\kappa_0\neq0$ are called caustics.

Before we define the oscillatory functions, we discuss the parametrization of $\lazero$ near the diagonal. For any $z_0\in \intx$, there is an open neighborhood $U_0$ such that the distance function $r(z_0, z)$ is smooth on $U_0\backslash \{z_0\}$. It is proved in the Appendix that the injectivity radius of $(\intx, g)$ is positive. So there exists an $\epsilon_1 > 0$ such that $r$  is smooth in an open neighborhood $\{0 < r < \epsilon_1\}$ of $\intd = \diag\backslash \p\diag$. In this neighborhood, the situation is exactly the same as in \cite{MSV} where the distance function is globally smooth away from $\intd$. Moreover, near $\Omega$ we have
\beq
\La\backslash \Omega = \{(z, z'; d_z r, -d_{z'}r) : z, z' \in \intx, z\neq z'\}.
\eeq 
So the Lagangian can be parametrized by the distance function near $\intd$. When lifted to $\xo$, $\beta_0^*r$ has a smooth extension to $\ff$, and vanishes quadratically at $\diag_0$, see \cite{MSV} and the introduction of \cite{Pre}. Despite of its singularity at $\diag_0$, it is preferable to take the distance function $\beta_0^*r$ as the phase function in a neighborhood of $\diag_0$, because it appears when we construct the semi-classical parametrix near the semi-classical front face $\mcs$ in Section 3. It is worth mentioning that,  as done in \cite{MSV}, we can blow up $X\times_0X$ along $\diag_0$ to get a new space $X\times_1 X$. Let $\beta_1: X\times_1 X\rightarrow X\times X$ be the blow down map, and $D$  be the new boundary face. Then $\beta_1^*r$ is a smooth boundary defining function of face $D$ in $X\times_1 X$.

Denote by $\mcv_b(\Xh)$ the set of smooth vector fields tangent to  $\mcl,\mcr,\mca$ faces of $\Xh$. In the following definition, we need a conormal function space as in \cite{MSV}, see also \cite{MM},
\beq
\mck^{a, b, c}(\Xh) = \{u \in C^\infty(\Xh) : V_1V_2\cdots V_m u \in \rho_\mcl^a\rho_\mca^b\rho_\mcr^c\rho_\mcs^{-n-1}C^\infty(X_{0,\hbar}^2), V_i \in \mcv_b(\Xh), m\in \mn\}. 
\eeq
To simplify things, we also say $u\in \mck^{a, b, c}(\Xh)$ if the lift of $u$ to $\Xh$ belongs to $\mck^{a, b, c}(\Xh)$.

\begin{definition}\label{osc}
Let $(\intx, g)$ be a $n+1$ dimensional non-trapping asymptotically hyperbolic manifold, and $\La$ be the Lagrangian submanifold defined in \eqref{defla}.  An oscillatory function $u(m, \sigma, h)$ on $\text{int}(\xo)$ associated to $\lazero = \beta_0^*\La$ belongs to  $I_\hbar^{a + i\soh, \mu - \kappa, c+ i\soh}(X, \La)$, if it is a locally finite sum of oscillatory integrals  
\begin{equation*}
I(m, h, \sigma) = \int e^{-i\frac{\sigma}{h}\phi(m, \theta)} b(m,  \theta, h, \sigma) d \theta, \ \ h\in(0, 1), \sigma \in \Omega_\hbar,
\end{equation*}
and that 
\beqq\label{equ}
u \in \rho_L^{i\soh}\rho_R^{i\soh}  \mck^{a, \mu-\kappa, c}(\Xh).
\eeqq
Here away from $\diag_0$, the phase function $\phi$ is as in Proposition \ref{lapara}, and near $\diag_0$, $\phi = \beta_0^*r$. The amplitude function $b(m, \sigma,  \theta, h)\sim \sum_{j = 0}^{\infty} b_j(m, \sigma,  \theta) h^{\mu - \ha N_0 + j}$ is compactly supported in $\theta$, where $N_0$ is the rank of  $d_{m, \theta} d_\theta \phi$ at $d_\theta\phi = 0$. The constant $\kappa \geq 0$ is determined by the caustics of $\La_0$. 
\end{definition}

We remark that the factor $\rho_L^{i\soh}\rho_R^{i\soh}$ in \eqref{equ} comes from the asymptotics of the phase function. The amplitude function may be singular at $\diag_0$. 
Finally, we define 
\beq
\Psi_{0, \hbar}^{m, a + i\soh, \mu - \kappa, c + i\soh}(X, \La) = \Psi_{0, \hbar}^m(X) + I_\hbar^{a + i\soh, \mu-\kappa, c + i\soh}(X, \Lambda),
\eeq
meaning the collection of operators $P = P_1 + P_2$ such that $P_1\in\Psi_{0, \hbar}^m(X)$ and the lift of the kernel $\beta_0^* K_{P_2}\in I_\hbar^{a + i\soh, \mu - \kappa, c + i\soh}(X, \Lambda)$. This is the space where our parametrix belongs to.

\section{Semiclassical Parametrix and Resolvent Estimates}

\subsection{Construction of a parametrix. }
In this section, we prove Theorem \ref{main0}. That is for $h\in (0, 1), \sigma \in \Omega_\hbar$, there exist  two operators $
G(h, \sigma)\in \Psi_{0,\hbar}^{-2, \fnt + i\soh, -\frac{n}{2} - 1 - \kappa, \frac{n}{2} + i\soh }(X, \La)$  and $E(h, \sigma)$ with $\beta_\hbar^*E \in  \rho_\mcs^\infty \rho_\mca^\infty \rho_\mcf^\infty \rho_\mcl^\infty \rho_\mcr^{\frac{n}{2} + i\frac{\sigma}{h}}C^\infty(X^2_{0, \hbar})$, both of which are holomorphic in $\sigma$ and such that 
\begin{equation*}
P(h, \sigma)G(h, \sigma) = \Id + E(h, \sigma).
\end{equation*}

The strategy of the proof is the same as in  Melrose, S\'a Barreto and Vasy \cite{MSV} by successively removing singularities of the resolvent kernel from $\diag_\hbar$ and asymptotics at boundary faces $\mcs, \mca, \mcf$ and $\mcl$. We refer to the original paper for full details. The novelty here is at the semiclassical face $\mca$. To make it clear, we divide the proof to five subsections. 

\subsubsection{At $\diag_\hbar$.}

We begin by looking for the normal operators of $P(h, \sigma)$ at $\mcf$ and $\mcs$ faces. Near $\p X$  where the product decomposition \eqref{prod} is vaild, we use $(x, y)$ as the local coordinates for $X$. The Laplace-Beltrami operator is
\begin{equation}\label{eq4.1.1}
\lap_g = -(x\frac{\p}{\p x})^2 +nx\frac{\p}{\p x} - x^2\gamma\frac{\p}{\p x} + x^2\lap_H, 
\end{equation}
where $\gamma = \p_x\sqrt{\det H} /\sqrt{\det H}$, and $\lap_H$ is the positive Laplacian on $\p X$ parametrized by $x$. To get rid of the first order term, we conjugate $P(h, \sigma)$ by $x^{\frac{n}{2}}$ to get
\begin{equation*}
\begin{split}
Q(h, \sigma) \doteq x^{-\frac{n}{2}}P(h, \sigma)x^{\frac{n}{2}} 
& = h^2(-(x\p_x)^2 -\frac{n}{2}x\gamma -x^2\gamma\p_x + x^2\lap_H) - \sigma^2. 
\end{split}
\end{equation*}

Now we find the lift of the operator on $\Xh$. First, consider the $0$-blow up. Let $(x', y')$ be the local coordinates of the right factor of $X\times X$ near $\p X$. Then the center of the $0$-blow up is $\{ x = x' = 0, y = y' \}$. Near $\ff$ and $L$, we can use projective coordinates
\begin{equation}\label{cord1}
X = x/x', \ \ Y = (y - y')/x', \ \ x',\ \ y'.
\end{equation} 
Here $X$ is a boundary defining function for $L$ and $x'$ is a boundary defining function for $\ff$. From $x\p_x = X\p_X$, we find that the lift of $Q(h, \sigma)$ is
\begin{equation*}
\begin{split}
\beta_0^* (Q(h, \sigma)) & = h^2(-(X\p_X)^2 - \frac{n}{2}X x' \gamma - X^2 x' \gamma \p_X + X^2\lap_{H(x' X, y' + x' Y)}(D_Y)) - \sigma^2. 
\end{split}
\end{equation*}
Here $\lap_H(D_Y)$ means the derivatives in $\lap_H$ are in $Y$ variable. When restricted to $\ff = \{x' = 0\}$,  the normal operator is
\begin{equation}\label{eqnf}
\mathcal{N}_\ff(Q) = h^2(-(X\p_X)^2 + X^2\lap_{H(0, y')}(D_Y)) - \sigma^2,
\end{equation}
which is the (semi-classical) Laplace operator on the fiber space 
\beq
\{X \geq 0, Y\in \mr^n\} \text{ with } \ g_0 = \frac{dX^2 + H(0, y', dY)}{X^2}
\eeq
parametrized by $y'$. 

Next consider the semi-classical blow up. The center is given by $\{X = 1, Y = 0, h = 0\}.$ We use projective coordinates,
\begin{equation}\label{cords}
h,\ \ X_\hbar = \frac{X - 1}{h},\ \ Y_\hbar = \frac{Y}{h},\ \ x',\ \  y'.
\end{equation}
Then $\p_X = \frac{1}{h}\p_{X_\hbar}, \p_Y = \frac{1}{h}\p_{Y_\hbar}$, and the lift of $Q$ is
\begin{equation*}
\begin{split}
\beta^*_\hbar(Q(h, \sigma)) 
&= -((1 + hX_\hbar)\p_{X_\hbar})^2 -\frac{n}{2}h^2(1+ hX_\hbar) x' \gamma - h(1+ hX_\hbar)^2 x' \gamma\p_{X_\hbar} \\
&+ (1+ hX_\hbar)^2\lap_{H( x' (1+ hX_\hbar), y' + x' hY_\hbar)}(D_{Y_\hbar}) - \sigma^2.
\end{split}
\end{equation*}
When restricted to the front face $\mcs = \{h = 0\}$, the normal operator is
\begin{equation*}
\mathcal{N}_\mcs(Q(h, \sigma)) = -\p_{X_\hbar}^2 + \lap_{H(x', y')}(D_{Y_\hbar}) - \sigma^2,
\end{equation*}
which is the Laplacian on the fiber space 
\beq
(X_\hbar, Y_\hbar)\in \mr^{n+1} \text{ with } g_{e} = dX^2_\hbar + \sum_{i, j = 1}^n H_{ij}(x', y') dY_{\hbar, i}dY_{\hbar, j},
\eeq
parametrized by $(x', y')$. If $\sigma \in \Omega_\hbar$,  only $(\re\la)^2$ shows up in the normal operator. But this is not an important issue.

It remains to find the normal operator at $\mcs$ but away from $\mcf$. Here the conjugation by $x^{\fnt}$ does not play a role. Let $z$ be the local coordinates of $X$, and the metric $g = \sum_{i, j = 1}^{n + 1}g_{ij}dz_idz_j$. The center of the semi-classical blow up is $\{ z = z', h = 0\}$. We take the projective coordinates
\beq
Z_\hbar = \frac{z - z'}{h}, \ \ h,\ \  z',
\eeq 
which is valid over the interior of the semi-classical front face $\mcs$. Then
\begin{equation*}
\beta_\hbar^*(P(h, \sigma)) = -\frac{1}{\sqrt{\det g}}\frac{\p}{\p Z_{\hbar, i}}(\sqrt{\det g} g^{ij}\frac{\p}{\p Z_{\hbar, j}}) -\frac{n^2}{4}h^2 -\sigma^2.
\end{equation*}
When restricted to $\mcs = \{h = 0\}$, 
\begin{equation*}
\mathcal{N}_\mcs(P) = -\frac{1}{\sqrt{g}}\frac{\p}{\p Z_{\hbar, i}}(\sqrt{g} g^{ij}\frac{\p}{\p Z_{\hbar, j}}) - \sigma^2,
\end{equation*}
which is the Laplacian on the fiber space $Z_\hbar \in \mr^{n+1} \text{ with } g_e = g(z'), $ parametrized by $z'$.

Finally, it is evident that $\beta_\hbar^*Q$ is an elliptic operator uniformly up to $\mcf$ and $\mcs$ faces. By doubling the manifold $X_{0, \hbar}^2$ across the $\mcf$ and $\mcs$ faces, we get an elliptic differential operator of order $2$ on a manifold without boundary. By the standard parametrix construction of elliptic operators on compact manifold without boundary, there exist $G_0(\sigma)\in\Psi_{0, \hbar}^{-2}(X)$ and  $E_0(\sigma)\in \Psi_{0, \hbar}^{-\infty}(X)$ holomorphic in $\sigma$ such that
$P(h, \sigma)G_0(\sigma) = \Id + E_0(\sigma).$ Moreover, $G_0(\sigma), E_0(\sigma)$ can be arranged to be supported near $\diag_\hbar$.  

\subsubsection{At $\mcs$ face.}
We look for $G_1(\sigma)$ such that 
\beq
P(h, \sigma)G_1(\sigma) - E_0(\sigma) =  E_1(\sigma) \in \rho_\mcs^\infty C^\infty(X^2_{0, \hbar}) .
\eeq
The idea is that since $E_1$ is smooth up to $\mcs$, we can write down the Taylor expansion of $E_1$ in $\rho_\mcs$ and remove the coefficients using the normal operator $\mathcal{N}_\mcs(Q)$. This can be carried out in a sufficiently small neighborhood of $\mcs$, where the lift of the geodesic distance function $r$ is well defined. Therefore, the same construction in Section 5 of \cite{MSV}, which uses essentially the distance function, works here. The result is that we can find $G_1, E_1$ whose Schwartz kernels satisfy
\beq
G_1 \in e^{-i\frac{\sigma}{h}r}\rho_\mcs^{-n-1}\rho_\mca^{-\fnt - 1}C^\infty(\Xh),\ \ E_1 \in e^{-i\soh r}\rho_\mcs^\infty\rho_\mca^{-\fnt}C^\infty(X^2_{0, \hbar}),
\eeq
and are supported in a neighborhood of $\diag_\hbar$  away from $\mcl, \mcr$ faces. Notice that the center of the semi-classical blow up is given by $\{\beta_0^*r = 0, h = 0\}$. So near the corner $\mcs\cap \mca$, we can use projective coordinates and set $\rho_\mcs = \beta_0^*r, \rho_\mca = h/\beta_0^*r$. Therefore, it becomes clear that $G_1 \in I_\hbar^{\infty, -\fnt - 1, \infty}(X, \La)$.

\subsubsection{At $\mca$ face.}
Since $E_1$ vanishes to infinite order at $\mcs$, we may blow down $\Xh$ to $\xo\times [0, 1)$ and regard $E_1$ as in $I_\hbar^{\infty, -\fnt - \kappa, \infty}(X, \La)$. 
We can take $\rho_\mca = h$, and to remove the asymptotics of $E_1$ at $\mca$ is equivalent to remove the coefficients of Taylor expansions  of $E_1$ in $h$ to $O(h^\infty)$. The main result is
\begin{lemma}\label{geo}
For $h\in (0, 1), \sigma \in \Omega_\hbar$, there exists $G_2 \in I_\hbar^{\fnt + i\soh, -\fnt-1-\kappa, \fnt + i\soh}(X, \Lambda)$ holomorphic in $\sigma$ such that 
\begin{equation}\label{approx}
P(h, \sigma)G_2 - E_1 = E_2 \in  h^\infty I_\hbar^{\fnt+i\soh, \infty, \fnt + i\soh}(X, \Lambda).
\end{equation}
\end{lemma}

\begin{proof}
We follow the classical geometric optics method. On $\xo\times [0, 1)$, we know that $E_1 = h^{-\fnt}  e^{-i\soh r}\tilde E_1$, where $\tilde E_1$ is supported in a  small neighborhood of $\diag_0$ and has an asymptotic expansion in $h$. Formally, we can write $E_1$ as a locally finite sum of oscillatory integrals 
\begin{equation*}
E_1 = \sum_{l = 1}^{M_\phi} \int_{\Theta_l} e^{-i\frac{\sigma}{h}\phi_l(m, \theta)} b_l(m, \theta, \sigma, h) d\theta, \ \ m \in U_l, \theta\in \Theta_l, 
\end{equation*}
where the phase functions $\phi_l$ are as in Definition \ref{osc}, $M_\phi$ is the number of phase functions, $U_l$ forms an open cover of $\intxo$, $\Theta_l$ are bounded sets in $\mbr^{M_l}$ and the amplitude $b_l(m, \theta, \sigma, h)\sim \sum_{j = 0}^\infty b_{l, j}(m, \theta, \sigma)h^{-\fnt - \frac{N_l}{2} + j}$ are compactly supported in $\theta$. In fact, $b_l$ vanish if $\phi_l \neq \beta_0^*r$. Our task is to find 
\begin{equation}\label{eqg2}
G_2 = \sum_{l = 1}^{M_\phi} \int_{\Theta_l} e^{-i\frac{\sigma}{h}\phi_l(m, \theta)}a_l(m, \theta, \sigma, h) d\theta,
\end{equation}
where $a_l(m, \theta, \sigma, h) \sim \sum_{j = 0}^{\infty} a_{l,j}(m, \theta, \sigma) h^{-\fnt - 1 - \frac{N_l}{2} + j}$ compactly supported in $\theta$, such that 
\begin{equation}\label{eqints}
\sum_{l = 1}^{M_\phi}\int_{\Theta_l} ( P(h, \sigma) e^{-i\frac{\sigma}{h}\phi_l(m, \theta)}a_l(m, \theta, \sigma, h) - e^{-i\frac{\sigma}{h}\phi_l(m, \theta)}b_l(m, \theta, \sigma, h) )d \theta = O(h^\infty).
\end{equation}

From the standard geometric optics method, we find with some calculation that over $\intx\times \intx$, the phase functions $\phi_l$ should satisfy the eikonal equation $p(z, d_z\phi_l) -\ha = 0$ for $p(z, \zeta) = \ha |\zeta|_{g^*}^2$. The $a_{l, j}$ should satisfy transport equations along the integral curves of the Hamilton vector field 
\begin{gather}
-2H_p(z, d_z \phi_l) a_{l, 0} + (\lap_g \phi)a_{l, 0} = b_{l, 0}, \label{trans1}\\
-2H_p(z, d_z \phi_l) a_{l, j} + (\lap_g \phi)a_{l, j} = \frac{i}{\sigma}(\lap_g - \frac{n^2}{4}) a_{l, j-1} + b_{l, j}, \ \ j = 1, 2, \cdots. \label{trans2}
\end{gather}
Here $H_p(z, d_z \phi_l)$ denotes the restriction of $H_p$ to the Lagrangian. Actually, we consider solving the equations on $\xo$, hence all of these equations should be lifted to $\xo$. 

Since $\phi_l$ parametrizes $\lazero$ locally,  they satisfy the lifted eikonal equation at the critical sets  $C_l \doteq \{(m, \theta) \in U_l\times \Theta_l : d_\theta \phi_l = 0\}$. The transport equations are first order linear ODEs along the integral curves of $\beta_0^*H_p$. To solve them, we regard the $b_{l, j}$ as globally defined functions $\hat b_j$ on $\La_0$. This is because $b_{l, j}$ is well defined on $\La_0$ where $\phi_l = \beta_0^*r$ and they vanish otherwise. By imposing zero initial conditions at $\diag_0$, the equations can be solved holomorphically in $\sigma$ to get $\hat a_j$ globally defined on $\La_0$. From the non-trapping assumption, the integral curves approach the left and right faces of $\xo$.  We now determine the asymptotics of $\hat a_j$   near the left face $L$. It suffices to do the computations locally in the following three types of regions which cover a neighborhood of $L$.

{\em Region 1: Near $L$ and away from $\ff$.} This avoids the $0$-blow up, and we can take $(x, y, z')$ as the local coordinates near  $\p X\times X$. The Laplace operator is found in \eqref{eq4.1.1}. From Proposition \ref{lapara}, the phase function in such a region can be written as $\phi_l(x, y, z', \theta) = -\log x + F(x, y, z', \theta)$ with $F$ smooth up to $\p X\times X$. Then $\lap_g \phi_l = -n + xC^\infty$, by which we mean that $\lap_g\phi_l = -n + x \tilde F$ with $\tilde F$ smooth where in concern. This abbreviation is used throughout the rest of the proof to simplify notations. 

We can use the product structure of $X\times X$ and consider the Hamiltonian as $p(x, y, \xi, \eta) = \fot(x^2 \xi^2 + x^2 h(x, y, \xi, \eta))$. Then
\beq
H_p  = \xi x^2 \p_x - (x\xi^2 + x h + \fot x^2 \p_x h)\p_\xi + \fot x^2 H_h.
\eeq
On $\La$, we have $\xi = \p_x \phi_l = -1/x + C^\infty$. Therefore,
$
-2 H_p (z, d_{z}\phi_l ) = 2 x \p_x   + x C^\infty.
$
Locally, we write $\hat a_j = a_{l, j}$, hence the first transport equation \eqref{trans1} becomes
\beq
-2 H_p a_{l, 0} + (\lap_g \phi_l) a_{l, 0} = 2 x \p_x a_{l, 0} - n a_{l, 0} + xa_{l, 0} C^\infty = 0,
\eeq
because $E_1$ is supported away from $x = 0$. By a simple indicial analysis, we conclude that $a_{l, 0} = x^{\fnt} C^\infty$.  For $a_{l, 1}$, the right hand side of \eqref{trans2} is
\beq
\frac{i}{\sigma}(\lap_g - \frac{n^2}{4}) a_{l, 0} = x^{\fnt + 1} C^\infty.
\eeq
So the same argument can be repeated to find $a_{l, j}  = x^{\fnt} C^\infty$.  

{\em Region 2: Near $L\cap \ff$ and away from $R$.} Here we can use projective coordinates \eqref{cord1} to get 
\beq
\lap_g = -(X\frac{\p}{\p X})^2 +nX\frac{\p}{\p X} - (Xx')X\gamma\frac{\p}{\p X} + (Xx')^2\lap_H.
\eeq
Since $\beta_0$ is a diffeomorphism over the interior of $\xo$, it induces a symplectmorphism, which in local coordinate \eqref{cord1} is
\begin{gather*}
\beta_0^* : (x, y, x', y'; \xi, \eta, \xi', \eta')\in T^*(\intx\times\intx)\rightarrow (X, Y, x', y'; \la, \mu, \la', \mu') \in T^*(\xo),\\
\text{where }\la = x'\xi, \ \ \mu = x'\eta, \ \ \la' = \xi' + X\xi + \eta Y, \ \ \mu' = \eta + \eta'.
\end{gather*}
So the lift of the Hamiltonian $p$ is $p_0 = \beta_0^*p =\ha( X^2\la^2 + X^2 h(Xx', y' + x' Y, \mu) ),$ and
\beq
H_{p_0} = \beta_0^*H_p = \la X^2 \p_X - (X\la^2 + X h + \fot X^2 \p_X h)\p_\la + \fot X^2 H_h. 
\eeq
From Proposition \ref{lapara}, the phase function $\phi_l = -\log X + C^\infty$, so that $\la = \p_X \phi_l = -1/X + C^\infty$. Therefore,
$-2 H_{p_0} (m, d_m\phi_l) = 2 X \p_X + X C^\infty.$ 
Since $\lap_g \phi_l = -n + X C^\infty$, the first transport equation becomes
\beq
-2 H_{p_0} a_{l, 0} + (\lap_g \phi_l) a_{l, 0} = 2 X \p_X a_{l, 0} - n a_{l, 0} + X a_{l, 0} C^\infty = 0,
\eeq
which implies that $a_{l, 0} = X^\fnt C^\infty$. The rest is similar to Region 1. 

{\em Region 3: Near $L\cap \ff \cap R$. }  Without loss of generality, assume $y_1 - y_1' \geq 0$. We can use the following projective coordinates for the $0$-blow up
\begin{equation}\label{cord2}
u = y_1 - y_1', \ \ w = \frac{x}{y_1 - y_1'}, \ \ w' = \frac{x'}{y_1 - y_1'},\ \ y'  \text{ and } Z_j = \frac{y_j - y_j'}{y_1 - y_1'}, \ \ j = 2, 3,\cdots n.
\end{equation}
Here $w, w'$ and $u$ are boundary defining functions for $L, R$ and $\ff$ respectively. The Laplace operator becomes
\beq
\lap_g =  -(w\frac{\p}{\p w})^2 +nw\frac{\p}{\p w} - u^2 w\gamma\frac{\p}{\p w} + u^2 w^2 C^\infty. 
\eeq
The induced symplectic change of variables is 
\begin{gather*}
\beta_0^* : (x, y, x', y'; \xi, \eta, \xi', \eta') \in T^*(\intx\times\intx) \rightarrow (u, w, w', Z, y'; \tau, \la, \la', \mu, \mu')\in T^*(\xo),\\
\text{where } \la = \xi u,  \ \ \la' = \xi' u, \ \ \mu_j =  u \eta_j Z_j,\ \  j = 2, 3, \cdots, n\\
 \mu' = \eta + \eta', \ \ \tau = \xi w + \xi' w' + \eta_1 + \sum_{j = 2}^{n}\eta_jZ_j.  
\end{gather*}
So the lift of $p$ becomes 
$p_0 = \fot w^2(\la^2 + h(wu, y, \eta)),$
and the Hamilton vector field is of the form
\beq
H_{p_0}  = \la w^2 \frac{\p}{\p w} + w C^\infty.
\eeq
From Proposition \ref{lapara}, we know that the phase function $\phi_l = -\log w - \log w' + C^\infty$. Then $\la  = \p_{w}\phi_l = -1/w + C^\infty$, and we have $-2H_{p_0}(m, d_m\phi_l)= 2w \p_{w} + w C^\infty.$
Again, $\lap_g \phi_l = -n + w C^\infty$. So \eqref{trans1} becomes
\beq
-2 H_{p_0} a_{l, 0} + (\lap_g \phi_l) a_{l, 0} =  2w \p_{w}a_{l, 0} - n a_{l, 0} + wa_{l, 0}C^\infty = 0.
\eeq
This gives $a_{l, 0} = w_1^\fnt C^\infty$. The rest is similar to Region 1.

Up to now, we have shown that $\hat a_j = \rho_L^\fnt C^\infty(\La_0)$. For the asymptotics in $\rho_R$, we just need to notice that $\hat a_{j}$  automatically satisfies the right transport equation. This is because  polar coordinates  $(r, \omega, z')$ is used in the local construction at face $\mcs$ (see \cite{MSV}), and $E_1$ is expressed in polar coordinates hence symmetric in $z, z'$ when blown-down to $\intx\times\intx\times[0, 1)$.  So the asymptotics of $\hat a_{j}$  at $R$ is the same as at $L$, i.e. $\hat a_j = \rho_L^\fnt \rho_R^\fnt C^\infty(\La_0)$.

By Remark \ref{rmkext}, we let $\tilde \phi_l$ be the restriction of $\phi_l$ to $U_l\times \tilde \Theta_l$ for $\tilde \Theta_l$ precompact in $\Theta_l$, such that $\La_{\tilde \phi_l} \subset \La_{\phi_l}$ and $\La_0$ is covered by $\La_{\tilde \phi_l}$.  Let $\chi_k$ be a  locally finite partition of unity subordinated to the open covering $\La_{\tilde \phi_l}$, and let $a_{l, j} = (\sum_{k} \chi_k) \hat a_j$ where the summation is over $k$ such that $\supp \chi_k \subset \La_{\tilde \phi_l} \subset \La_{\phi_l}$. In particular, $a_{l, j}$ vanishes on $\La_0\backslash \La_{\tilde \phi_l}$. We can extend $a_{l, j}$ to a sufficiently small neighborhood of $\La_{\tilde \phi_l}$ in $T^*(\xo)$, or equivalently a small neighborhood of $\tilde C_l \doteq \{(m, \theta) \in U_l\times \tilde \Theta_l : d_\theta \tilde \phi_l = 0\}$. Because $\tilde \Theta_l$ is precompact in $\Theta_l$, $a_{l, j}$ can be extended to a smooth function in $U_l\times \Theta_l$ and compactly supported in $\theta$. To see this will not affect the asymptotics of $a_{l, j}$ at the left and right faces, we can smoothly extend $\rho_L^{-\fnt}\rho_R^{-\fnt}a_{l, j}$ instead. Therefore, we get well defined amplitude functions $a_{l, j}(m, d_m\phi_l)$ with $a_{l, j} = \rho_L^\fnt \rho_R^\fnt C^\infty(\overline{U_l} \times \Theta_l)$ and compactly supported in $\theta$. Here we take the closure of $U_l$ in $\xo$ to emphasize that the functions are smooth to the boundary of $\xo$.

Finally, by the Borel lemma, we get the asymptotic summation $a_l \sim \sum_j a_{l, j}h^{-\fnt - 1 - \frac{N_l}{2} + j}$ such that $a_l = \rho_L^\fnt\rho_R^\fnt h^{-\fnt - 1 - \frac{N_l}{2}}C^\infty(U_l\times \Theta_l)$ and is compactly supported in $\theta$. The operator $G_2$ is well-defined as in \eqref{eqg2}. Applying the stationary phase method see e.g. Sogge \cite{So}, the oscillatory integral in $G_2$ with amplitude $a_l$ is of $O(h^{-\fnt - 1 - \kappa})$ locally uniformly for some $\kappa\geq 0$. The remainder term is of $O(h^\infty)$ because the amplitudes solve \eqref{eqints} on the critical sets while away from the critical sets the oscillatory integral is $O(h^\infty)$ by the stationary phase method.  This completes the proof of the lemma.
\end{proof}

\subsubsection{At $\mcf$ face.}
This step is the same as in \cite{MSV}. We regard $E_2$ as a function on $\xo\times [0, 1)$. Since $E_2$ is smooth at $\ff$, we can write $E_2$ in Taylor series 
\beq
E_2 \sim \sum_{j=1}^\infty \rho_\ff^j E_{2, j}, \text{ where } E_{2, j} \in  \rho_R^{\frac{n}{2} + i\frac{\sigma}{h}} \rho_L^{\frac{n}{2} + i\frac{\sigma}{h}}h^\infty C^\infty(\ff\times [0, 1)).
\eeq 
This is well-defined because $|\im\sigma/h| < C$, and $ \rho_R^{\frac{n}{2} + i\frac{\sigma}{h}}\rightarrow 0$ as $h \rightarrow 0$ and $\rho_R \rightarrow 0$. Now we look for $G_3 \sim \sum_{j = 1}^\infty \rho_\ff^j G_{3, j}$ such that 
\begin{equation*}
P(h, \sigma)G_3 - E_2  = E_3 \in h^\infty \rho_\ff^\infty C^\infty(X^2_{0, \hbar}).
\end{equation*}
We found the normal operator at $\ff$ in \eqref{eqnf}, which is the semi-classical Laplacian of a hyperbolic metric
\begin{equation*}
\mathcal{N}_\ff(Q_0) = h^2(-(X\p_X)^2 + X^2\lap_{H(0, y')}(D_Y)) - \sigma^2.
\end{equation*}
Appyling Proposition 6.15 of \cite{MM}, we can find $G_{3, j} \in h^\infty \rho_R^{\frac{n}{2} + i\frac{\sigma}{h}} \rho_L^{\frac{n}{2} + i\frac{\sigma}{h}}C^\infty(\ff\times [0, 1))$. To see this will not destroy the asymptotics at $h\rightarrow 0$, we can repeat the construction near $\mcs$ on $\mcf$ instead of on $X^2_{0, \hbar}$ as done in \cite{MSV}. By Borel summation, we get 
$
G_3 \in h^\infty \rho_R^{\frac{n}{2} + i\frac{\sigma}{h}} \rho_L^{\frac{n}{2} + i\frac{\sigma}{h}}C^\infty(\xo\times [0, 1)).
$
This can be  arranged as $G_3\in h^\infty I^{\fnt + i\soh, \infty, \fnt + i\soh}_\hbar(X, \Lambda)$. The remainder term is 
$
E_3 \in h^\infty\rho_\ff^\infty \rho_L^{\fnt + 1 + i\frac{\sigma}{h}} \rho_R^{\frac{n}{2} + i\frac{\sigma}{h}}C^\infty(\xo\times[0, 1)).
$

\subsubsection{At $\mcl$ face.}
Since $E_3$ vanishes to infinie order in $\rho_\ff$, we can regard $E_3$ as a smooth function on $X\times X\times[0, 1)$. Near $x = 0$, the Laplacian is
\begin{equation*}
P(h, \sigma) = h^2( -(x\frac{\p}{\p x})^2 +nx\frac{\p}{\p x} - x^2\gamma\frac{\p}{\p x} + x^2\lap_H - \frac{n^2}{4}) - \sigma^2.
\end{equation*}
The indicial operator is
\beq
I(P) = h^2( -(x\frac{\p}{\p x})^2 +nx\frac{\p}{\p x} - \frac{n^2}{4}) - \sigma^2.
\eeq
The indicial roots are $\alpha = \frac{n}{2} \pm i\frac{\sigma}{h}$. For $E_3 \in h^\infty \rho_\ff^\infty \rho_L^{\frac{n}{2} + 1 + i\frac{\sigma}{h}} \rho_R^{\frac{n}{2} + i\frac{\sigma}{h}} C^\infty(X\times X \times [0, 1))$, we can use a Taylor series argument to find $G_4 \in h^\infty \rho_\ff^\infty \rho_L^{\frac{n}{2} + i\frac{\sigma}{h}} \rho_R^{\frac{n}{2} + i\frac{\sigma}{h}}C^\infty(X\times X\times [0, 1))$ such that 
\begin{equation*}
P(h, \sigma)G_4 - E_3 = E_4 \in h^\infty\rho_\ff^\infty \rho_L^\infty \rho_R^{\frac{n}{2} + i\frac{\sigma}{h}}C^\infty(X\times X\times [0, 1)),
\end{equation*}
which is the final error. We can arrange $G_4\in h^\infty\rho_\ff^\infty I_\hbar^{\infty, \infty, \frac{n}{2} + i\frac{\sigma}{h}}(X, \Lambda)$.  

To summarize, we find $G = G_0 + \tilde G$, with $G_0 \in \Psi_{0, \hbar}^{-2}(X)$, and $\tilde G = G_1+G_2 + G_3 + G_4 \in I_\hbar^{\frac{n}{2} + i\soh, -\frac{n}{2} - 1 - \kappa, \frac{n}{2} + i\soh}(X, \Lambda)$, such that 
$
P(h, \sigma)G - \Id = E_4,
$
where $E_4 \in \rho_\mcs^\infty \rho_\mca^\infty \rho_\mcf^\infty \rho_\mcl^\infty \rho_\mcr^{\frac{n}{2} + i\frac{\sigma}{h}}C^\infty(X^2_{0, \hbar})$. This completes the proof of Theorem \ref{main0}.

\subsection{High energy resolvent estimates}
According to Theorem \ref{main0},  the parametrix $G = G_0 + \tilde G$, where $G_0 \in \Psi_{0, \hbar}^{-2}(X)$, and $\tilde G\in I_\hbar^{\fracnt + i\soh, -\fracnt -1 - \kappa, \fracnt + i\soh}(X, \Lambda)$. The Schwartz kernel of $G_0$ is a conormal distribution to $\diag_\hbar$ of order $-2$, vanishing to infinite order at $\mcl, \mca, \mcr$. We can double the manifold $\Xh$ across $\mcf, \mcs$ faces where $\diag_\hbar$ intersects transversally to get a compact manifold without boundary. By using the standard $L^2$ theory of pseudodifferential operators, we have 
\beq
\|\rho^a G_0 \rho'^b f\|_{\ltx} \leq C \|f\|_{\ltx},
\eeq
where $a, b\in \mr, a+b \geq 0$, $\rho$ is a boundary defining function of $\p X$ on $X$, and $\rho'$ is the same function on the other copy of $X$. 
The constant  $C$ does not depend on $h$. 

For $\tilde G$, we first notice that on any compact set $K \subset \intx\times \intx \backslash \diag^\circ$, the function $\chi_K \tilde G$ is smooth and bounded, where $\chi_K$ is a smooth cut-off function supported in $K$. By Schur's lemma, for $a, b\in \mr, \sigma \in \Omega_\hbar$, there exists a constant $C$ such that
\beq
\|\rho^a \chi_K \tilde G(h, \sigma) \rho'^b f\|_{\ltx} \leq C h^{-\fnt - 1 - \kappa} \|f\|_{\ltx}.  
\eeq

Next near $\diag_\hbar$, the kernel is 
\begin{equation*}
\rho^a \tilde G(h, \sigma) \rho'^b = \rho_\mca^{-\fracnt -1-\kappa} \rho_\mcs^{-n-1} F,
\end{equation*}
where $F$ is bounded. Using the projective coordinates for the semi-classical blow up: $\rho_\mcs = r, \rho_\mca = h/r$, the kernel is controlled by
\begin{equation*}
|\rho^a \tilde G(h, \sigma) \rho'^b| \leq C (\frac{h}{r})^{-\fracnt -1 - \kappa} r^{-n-1}. 
\end{equation*}
The volume form $dg(z') = C(n) r^ndrdw$. Since the injectivity radius of $(\intx, g)$ is positive (see Lemma \ref{inject} in Appendix), there exists an $\epsilon_1 > 0$ sufficiently small so that
\begin{equation*}
\int_{r < \epsilon_1} |\rho^a \tilde G(z, z', h, \sigma) \rho'^b| dg(z')  \leq  C h^{-\fracnt -1 - \kappa} \int_0^{\epsilon_1} r^{\fracnt + \kappa} dr \leq C h^{-\fracnt -1 - \kappa},
\end{equation*} 
where the constant $C$ is uniform in $z$. The $L^2$ estimate follows from Schur's lemma.

Finally, to deal with the estimates near $\mcl, \mcr$, we recall Lemma 6.2 from \cite{MSV}, see also \cite{Ma3}.
\begin{lemma}\label{est1}
Suppose for $C > 0$ the Schwartz kernel of $B$ trivialized by $dg(z')$ satisfies $|\beta_0^* B(z, z', h)| \leq C \rho_L^\alpha \rho_R^\beta. $ For $\alpha, \beta > \fracnt$, there exists some constant $C'$ such that $\|B f\|_{\ltx} \leq CC' \|f\|_{\ltx}.$
\end{lemma}

The boundary defining functions are lifted to 
$
\beta_\hbar^*\rho = \rho_\mcr \rho_{\mcf}, \ \ \beta_\hbar^*\rho' = \rho_\mcl \rho_{\mcf}.
$
From Theorem \ref{main0}, we have near $\mcl$ and $\mcr$ that 
\begin{equation*}
\rho^a \tilde G(h, \sigma) \rho'^b = h^{-\fracnt - 1 - \kappa} \rho_\mcf^{a+b} \rho_\mcl^{a + \fracnt + i\soh} \rho_\mcr^{b + \fracnt + i\soh} F,
\end{equation*}
where $F$ is smooth and bounded.  For $a + b \geq 0, a, b > \frac{\im\sigma}{h}$, we can apply Lemma \ref{est1} to get the $L^2$ estimates. Finally, by a finite partition of unity of $\Xh$, the local estimates give 
\begin{equation*}
\|\rho^a G(h, \sigma) \rho^b f\|_{\ltx} \leq C h^{-\fracnt - 1 - \kappa}\|f\|_{\ltx}. 
\end{equation*} 

The remainder term $E$ can be estimated in the same way. For any $M, N \geq 0$, we have
\begin{equation*}
\begin{split}
&|\rho^a E(h, \sigma) \rho'^b| \leq Ch^N \rho_\mcl^{a + M} \rho_\mcr^{\fracnt -\frac{\im\sigma}{h} + b}\rho_\mcf^{a + b}, \text{ near } \mcr, \mcl, \\
&|\rho^a E(h, \sigma) \rho'^b| \leq C (\frac{h}{r})^N r^M, \text{ near } \diag_\hbar. 
\end{split}
\end{equation*}
For $a + b \geq 0, b > \frac{\im\sigma}{h}, M > \fracnt - a$ and $N\in \mn$, there exists $C_N > 0$ such that 
\begin{equation}\label{eqeest}
\|\rho^a E(h, \sigma) \rho^b f\|_{\ltx} \leq C_N h^N \|f\|_{\ltx}.
\end{equation}

For the resolvent estimate, we first recall the relation 
\beqq\label{eqsemi}
P - \la^2 = h^2P(h, \sigma), \ \ h = 1/|\re\la|, \sigma =  \la / |\re\la|.
\eeqq
For $\la/|\re\la| \in\Omega_\hbar$ and $ \im \la < 0$, the resolvent is holomorphic. We can apply $R(\la)$ to \eqref{eqsemi} and use Theorem \ref{main0} to get $h^{-2}G(h, \sigma) = R(\la)(\Id + E(h, \sigma)),$ and with weight
\begin{equation}\label{eq2}
h^{-2}\rho^a G(h, \sigma)\rho^b = \rho^a R(\la)\rho^b(\Id + \rho^{-b} E(h, \sigma)\rho^b).
\end{equation}
For $h$ sufficiently small, $\Id + \rho^{-b} E(h, \sigma)\rho^b$ is invertible by  \eqref{eqeest}. Therefore, we can extend $R(\la)$  holomorphically to 
\beq
h\in (0, 1/C_2), \ \ |\im\sigma| < C_1 h \Longleftrightarrow |\re\la| > C_2,\ \  |\im\la| < C_1,  
\eeq
 through \eqref{eq2}. The high energy estimates follow  from the estimates of $G$ and $E$. This finishes the proof of Theorem \ref{main1}.  

\begin{remark}
If $a, b = \im\la$, one can obtain resolvent estimates on logarithmic weighted $L^2$ spaces by using the full version of Lemma \ref{est1}.
\end{remark}

\section{Exponential Decay of the Radiation Field}

\subsection{Analysis of the Eisenstein funciton.}
The radiation field $\mcr_{\pm}$ is related to the transposed Eisenstein function, which we study first in this section. By fixing a product decomposition (\ref{prod}), the choice of the boundary defining function $x$ fixes a metric $H_0$ on $\p X$ out of the conformal class determined by $g$. As pointed out in S\'a Barreto \cite{Sa}, the transposed Eisenstein function
\begin{equation}\label{eisen}
\mce(\la) = x^{-\fracnt - i\la}R(\la)\mid_{x = 0},
\end{equation}
is a well-defined operator whose kernel is a distribution on $\p X\times X$, conormal to $\p\diag = \{(z, z')\in \p X\times \p X : z = z'\}$. 
The proof essentially follows from Proposition 4.1 of \cite{JS1}. The singularities of the kernel $\mce(\la, y, z')$ can be resolved on a blown-up space $\p X\times_0 X$   defined by blowing up $\p X\times X$ along $\p\diag$. Let $\tilde \beta_0: \p X\times_0 X \rightarrow \p X\times X$ be the blow down map. It is clear that $\tilde\beta_0 = \beta_0\mid_{\p X\times_0 X}$. The new manifold has two boundary faces. The front face $\tilde\ff$ introduced by the blow up is the closure of $\tilde \beta_0^{-1}(\p\diag)$. This is a half-sphere bundle over $\p\diag$. The right face $\tilde R$ is the closure of $\tilde \beta_0^{-1}(\p X\times \intx)$. See Figure \ref{fig6}. 

\begin{figure}[bhtp]
\centering
\scalebox{0.9}
{\begin{pspicture}(0,-2.74)(12.66,2.72)
\psline[linewidth=0.04cm,arrowsize=0.05291667cm 2.0,arrowlength=1.4,arrowinset=0.4]{->}(2.26,-0.3)(5.26,-0.3)
\psline[linewidth=0.04cm,arrowsize=0.05291667cm 2.0,arrowlength=1.4,arrowinset=0.4]{->}(1.4,-0.68)(0.0,-2.68)
\psline[linewidth=0.04cm,arrowsize=0.05291667cm 2.0,arrowlength=1.4,arrowinset=0.4]{->}(1.66,-0.1)(1.66,2.7)
\psline[linewidth=0.04cm](2.26,-0.32)(3.68,-1.72)
\psline[linewidth=0.04cm](1.4,-0.7)(2.82,-2.26)
\psline[linewidth=0.04cm,linestyle=dotted,dotsep=0.16cm](1.66,0.14)(1.66,-0.38)
\psline[linewidth=0.04cm,linestyle=dotted,dotsep=0.16cm](1.66,-0.3)(2.26,-0.3)
\psline[linewidth=0.04cm,linestyle=dotted,dotsep=0.16cm](1.66,-0.3)(1.4,-0.66)
\psarc[linewidth=0.04](1.85,-0.53){0.47}{27.645975}{207.47443}
\psarc[linewidth=0.04](3.24,-2.0){0.5}{32.47119}{206.56505}
\usefont{T1}{ptm}{m}{n}
\rput(4.84,-1.3){$\tilde R$}
\usefont{T1}{ptm}{m}{n}
\rput(2.45,-1.395){$\tilde{\text{ff}}$}
\psline[linewidth=0.04cm,arrowsize=0.05291667cm 2.0,arrowlength=1.4,arrowinset=0.4]{->}(9.28,-0.3)(12.46,-0.3)
\psline[linewidth=0.04cm,arrowsize=0.05291667cm 2.0,arrowlength=1.4,arrowinset=0.4]{->}(9.28,-0.3)(8.02,-2.6)
\psline[linewidth=0.04cm,arrowsize=0.05291667cm 2.0,arrowlength=1.4,arrowinset=0.4]{->}(9.28,-0.3)(9.28,2.68)
\psline[linewidth=0.04cm,arrowsize=0.05291667cm 2.0,arrowlength=1.4,arrowinset=0.4]{->}(5.62,2.1)(7.36,2.1)
\usefont{T1}{ptm}{m}{n}
\rput(6.4,2.525){$\tilde\beta_0$}
\usefont{T1}{ptm}{m}{n}
\rput(10.89,1.045){$X$}
\usefont{T1}{ptm}{m}{n}
\rput(12.32,-0.775){$y'$}
\usefont{T1}{ptm}{m}{n}
\rput(8.57,-2.515){$y$}
\usefont{T1}{ptm}{m}{n}
\rput(9.78,2.265){$x'$}
\usefont{T1}{ptm}{m}{n}
\rput(8.12,-1.135){$\p X$}
\end{pspicture} }
\caption{$\p X\times_0 X$ constructed from $\p X\times X$}
\label{fig6}
\end{figure}
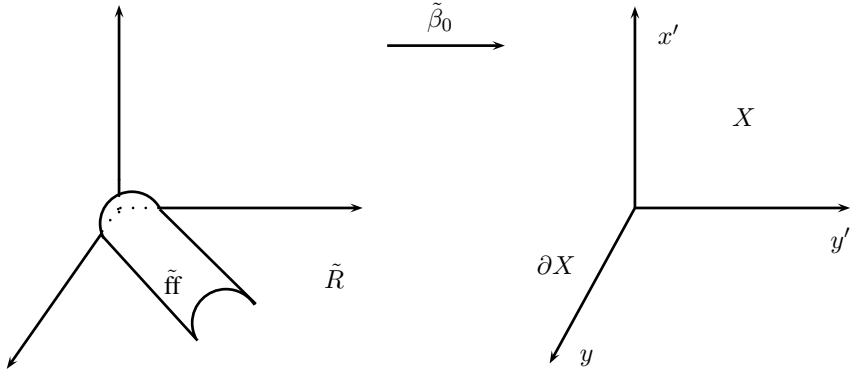

The definition \eqref{eisen} of transposed Eisenstein function clearly depends on the choice of the boundary defining function $x$, and can be made independent by working with density bundles as  in  \cite{JS1}. However, we shall take $\rho = x$ in this section for simplicity.  

We  study $\mce(\la)$ in the limit $|\re\la|\rightarrow \infty$ or equivalently as $h\rightarrow 0$ with $h = 1/|\re\la|$. From  \eqref{eq2} and \eqref{eisen}, we can write
\begin{equation}\label{eiseneq}
\mce(\la)\rho'^b = \rho^{-\fracnt - i\la}R(\la) \rho'^b \mid_{\rho = 0} = h^2\rho^{-\fracnt - i\soh} G(h, \sigma)\rho'^b \mid_{\rho = 0} \circ (\Id + \rho'^{-b} E(h, \sigma)\rho'^b)^{-1}.
\end{equation}
For $h$ sufficiently small, the operator $(\Id + \rho^{-b} E(h, \sigma)\rho^b)^{-1}$ is bounded on $\ltx$ uniformly in $h$. For the $L^2$ estimates of $\mce(\la)$, it suffices to study
\begin{equation}\label{neweis}
\mcep(h, \sigma) \doteq \rho^{-\fracnt - i\soh} G(h, \sigma) \mid_{\rho = 0},  
\end{equation}
for $h \in (0, 1)$  and $\sigma\in \Omega_\hbar$. Then,  
\beqq\label{mce}
\mce(\la) \rho'^b = h^2 \mcep(h, \sigma)\rho'^b \circ (\Id + \rho'^{-b} E(h, \sigma)\rho'^b)^{-1}.
\eeqq

The Schwartz kernel of $\mcep(h, \sigma)$ can be resolved on $\p X\times_0 X\times [0, 1)$. Actually, from Theorem \ref{main0}, we know that $G = G_0 + \tilde G$ and they satisfy
\begin{equation}\label{eq62}
\begin{split}
&\beta_\hbar^*(\rho^{-\fracnt - i\la} G_0)\mid_{\mcr} = 0,\\
&\beta_\hbar^*(\rho^{-\fracnt - i\la} \tilde G)\mid_{\mcr} = \rho_\mcf^{-\fracnt - i\frac{\sigma}{h}} \rho_\mcl^{\fracnt + i\frac{\sigma}{h}} \rho_\mca^{-\fracnt - 1 - \kappa}\rho_\mcs^{-n - 1}F\mid_{\mcr}.
\end{split}
\end{equation}
It is clear that $F|_\mcr$ is smooth on $\mcr$. Now we consider the restriction of  boundary defining functions $\rho_\mcf, \rho_\mca, \rho_\mcl, \rho_\mcs$ to $\mcr$. This can be done in convenient local coordinates near the blow-up. 

We use $(y, x', y')$ as the local coordinates of $\p X\times X$ near the boundary. The center of the blow up $\p\diag$ is $\{x' = 0, y = y'\}$. First over the interior of $\tilde\ff$, we use projective coordinate
\begin{equation}\label{D1}
x',\ \ Y = \frac{y - y'}{x'},\ \ y.
\end{equation}
Then $x'$ is a boundary defining function of $\tilde\ff$, and \eqref{D1} is just the restriction of (\ref{cord1}) to $X = 0$.

Next, near the corner  $\tilde\ff \cap \tilde R$, we take $y_1 - y_1' > 0$ and use projective coordinate
\begin{equation}\label{D2}
t = y_1 - y_1',\phantom{a} s = \frac{x'}{y_1 - y_1'},\phantom{a} Z_j = \frac{y_j - y_j'}{y_1 - y_1'},\phantom{a} y.
\end{equation}
Then $t, s$ are boundary defining functions of $\tilde \ff, \tilde L$ respectively, and \eqref{D2} is the restriction of \eqref{cord2} to $w = 0$. Since $\mcs\cap \mcr = \emptyset$, the boundary defining function $\rho_\mcs$ is smooth and positive on $\mcr$. Also, $\rho_\mca\mid_\mcr$ is a smooth  boundary defining function of the face $\{h = 0\}$ in $\p X\times_0 X\times [0, 1)$.

Now we see that the Schwartz kernel of $\mcep(h, \sigma)$ is well-defined on $\p X\times_0 X\times[0, 1)$, and
\begin{equation}\label{eq64}
\mcep(h, \sigma) = \beta_\hbar^*(\rho^{-\fracnt - i\soh}G)\mid_\mcr =  \rho_{\tilde\ff}^{-\fracnt - i\soh} \rho_{\tilde L}^{\fracnt + i\soh} h^{-\fracnt - 1 - \kappa} \tilde F(h, \sigma), 
\end{equation}
where $\tilde F$ is smooth on $\p X\times_0 X\times [0, 1)$. 
The $L^2$ estimate of $\mce_0(h, \sigma)$  can be  derived from Schur's lemma. As before, we use the volume form on $\intx$ to trivialize the density bundle. Hence for $T: C^\infty_0(X) \rightarrow C^{-\infty}(\p X)$ with Schwartz kernel  $K_T(y, z')dg(z')$, the action of $T$ is
\begin{equation*}
Tf(y) = \int K_T(y, z')f(z')dg(z'), \ \ \forall f\in C_0^\infty(X).
\end{equation*}

We start from the high energy estimate. 
\begin{prop}\label{higheis}
For any $C_2 > 0$, there exists some $C_1 > 0$ such that for $\{\la \in \mc : |\re\la| > C_1, |\im\la|< C_2\}$ and $b > \frac{n}{2} + |\im\la|$, the operator $\mce(\la)\rho^b$ is a holomorphic family of bounded operators from $\ltx$ to $L^2(\p X)$. Moreover, there exist $C>0, \kappa \geq 0$ such that
\begin{equation*}
\|\mce(\la)\rho^b f \|_{L^2(\p X)} \leq C |\la|^{\frac{n}{2} - 1 - \kappa} \|f\|_{\ltx }. 
\end{equation*}
\end{prop} 
\begin{proof}
The holomorphy follows from that of $R(\la)$. For $h = 1/|\re\la|$ sufficiently small and $\sigma = \la/|\re\la| \in \Omega_\hbar$, $(\Id + \rho^{-b}E(h, \sigma)\rho^b)^{-1}$ is bounded on $\ltx$ uniformly in $h$. Hence it suffices to estimate $\mcep(h, \sigma)$. To apply Schur's lemma, we need
\beq
\text{I} = \sup_{y\in \p X} \int |\mcep(y, z', h) \rho^b| dg(z') \leq C_1, \text{ and } \text{II} = \sup_{z'\in X} \int |\mcep(y, z', h) \rho^b | dH_0(y) \leq C_2.
\eeq
We can estimate the two integrals locally by a finite partition of unity of $\p X\times_0 X$. The case away from $x' = 0$ is straightforward because the kernel is of the form $h^{-\fnt - 1 - \kappa} F$ for some $F\in C_0^\infty(\p X\times X)$. Hence it remains to deal with the case near the blow up, which can be covered by the two types of local coordinates \eqref{D1}, \eqref{D2}. 

First in projective coordinate \eqref{D1}, the pull-back of the volume forms are
\begin{equation*}
\begin{split}
\tilde\beta_0^*(\frac{dx'dy'}{(x')^{n+1}}) = \frac{dx'dY}{x'},\ \ \tilde\beta_0^*(dy) = dy.
\end{split}
\end{equation*}
Here we absorbed the factors $\sqrt{\det g}$ and $\sqrt{\det H_0}$  to the kernel. The Schwartz kernel of $\mcep(h, \sigma)$ is  
\begin{equation*}
\tilde \beta_0^*\mcep(h, \sigma) = \rho_{\tilde\ff}^{-\fracnt - i\soh} \rho_{\tilde L}^{\fracnt + i\soh} h^{-\fracnt - 1 -\kappa}\tilde F(h, \sigma),
\end{equation*}
where $\tilde F$ is smooth on $\p X\times_0 X\times [0, 1)$. We find that for $b > \fnt - \frac{\im\sigma}{h}$
\begin{equation*}
\begin{split}
\text{I} 
& =  h^{-\fracnt - 1-\kappa}\sup_y  \int    (x')^{b -\frac{n}{2} - i\frac{\sigma}{h} - 1} \tilde F(h, \sigma) dx'dY \leq C h^{-\fnt - 1 - \kappa}.
\end{split}
\end{equation*}
Here the integration is for small $x'$, and the constant $C$ depends on $\Omega_\hbar$. For the second integral,  
\begin{equation*}
\text{II} = h^{-\fracnt - 1-\kappa} \sup_{x', Y} \int (x')^{b -\fracnt -i \frac{\sigma}{h}} \tilde F(h, \sigma) dy \leq C h^{-\fnt - 1 - \kappa},
\end{equation*}
if $b \geq \fracnt  - \frac{\im\sigma}{h}$. Next we estimate the integrals in coordinate \eqref{D2}. The pull-back of densities become
\begin{equation*}
\begin{split}
\tilde\beta_0^*(\frac{dx'dy'}{(x')^{n+1}}) = \frac{ds dt dz}{s^{n+1}t},\ \ \tilde\beta_0^*(dy) = dy.
\end{split}
\end{equation*}
Using the kernel asymptotics, we have for $b > \fnt + |\frac{\im\sigma}{h}|$ that
\begin{equation*}
\begin{split}
\text{I} &= h^{-\fracnt - 1-\kappa}\sup_{y}\int t^{-\fracnt - i\frac{\sigma}{h}} s^{\fracnt + i\frac{\sigma}{h}} (st)^{b} \tilde F(h, \sigma) \frac{dsdtdz}{s^{n+1}t}\\
& = h^{-\fracnt - 1-\kappa} \sup_{y} \int t^{b -\fracnt - i\frac{\sigma}{h} - 1} s^{b - \fracnt + i\frac{\sigma}{h} - 1}  \tilde F(h, \sigma) ds dt dz\\
& \leq C h^{-\fnt  - 1 - \kappa}.
\end{split}
\end{equation*}
Here the integration is for $t$ and $s$ close to $0$. For the other integral, we have
\begin{equation*}
\text{II} = h^{-\fracnt - 1-\kappa} \sup_{t, s, z} \int t^{-\fracnt -i \frac{\sigma}{h}}s^{\fracnt + i\frac{\sigma}{h}} (st)^b \tilde F(h, \sigma) dy \leq C h^{-\fnt  - 1 - \kappa}.
\end{equation*}
if $b \geq \fracnt +|\frac{\im\sigma}{h}|$. Finally, by Schur's lemma, 
\beq
\|\mcep(h, \sigma) \rho^b f\|_{L^2(\p X)} \leq C h^{-\fracnt - 1-\kappa} \|f\|_{\ltx}.
\eeq
This finishes the proof in view of \eqref{mce} and translating the estimates in terms of $\la$. 
\end{proof}

The low and intermediate $L^2$ estimate of the Eisenstein function basically follows from the work of Joshi-S\'a Barreto \cite{JS1} and Mazzeo-Melrose \cite{MM}. We now briefly review some results to give the estimate, and refer the readers to the original papers for more details.

Theorem 7.1 of \cite{MM} tells that the resolvent $R(\la)$ has a meromorphic continuation from $\im\la < -\fracnt$ to $\mc\backslash \frac{i}{2}\mn$, with residues of finite rank operators. In particular, $R(\la) = R'(\la) + R''(\la)$ where $R'(\la)$ is a $0$-pseudodifferential operator of order $-2$ vanishing to infinity order at $L, R$ of $X\times_0 X$, and the Schwartz kernel of $R''(\la)$ trivialized by $dg(z')$ is 
\beq
R''(\la) = \rho_L^{\fracnt + i\la} \rho_R^{\fracnt + i\la} F(\la), 
\eeq
where $F$ is smooth in $\xo$ and meromorphic in $\la$. Here we used a different spectral parameter and trivialization of the density bundle as in \cite{MM}. By the definition of Eisenstein function (\ref{eisen}),
\begin{equation*}
\mce(\la)\rho'^b = \rho^{-i\la -\fracnt}R(\la)(\rho')^b\mid_{\rho = 0} = \rho^{-i\la - \fracnt}R''(\la)(\rho')^b\mid_{\rho = 0}.
\end{equation*}

On the blown-up space $\p X\times_0 X$, the lift of the Schwartz kernel with the density factor is
\begin{equation*}
\begin{split}
\tilde\beta_0^*(\mce(\la)\rho^b) & = (\rho_R\rho_\ff)^{-\fracnt - i\la}\rho_R^{\fracnt + i\la} \rho_L^{\fracnt + i\la} \rho_L^{b - n - 1}\rho_\ff^{-1}F(\la)\mid_R\\
& = \rho_{\tilde\ff}^{b -\fracnt - i\la -1} \rho_{\tilde L}^{b -\fracnt + i\la - 1}F(\la)\mid_{R}. 
\end{split}
\end{equation*}

Similar to the proof of Proposition \ref{eisen},  $\mce(\la)\rho^b$ is bounded from $\ltx$ to $L^2(\p X)$ for $b > \fracnt + |\im\la|$, and $\la$ in a compact set of $\mc$ free of resonances. We summarize the results to 
\begin{prop}\label{loweis}
There exists $\delta_0 > 0$ such that the Eisenstein function (\ref{eisen}) is holomorphic on $\Omega_0 \doteq \{\la\in\mc \backslash 0:  |\im\la| < \delta_0\}.$ For $b > \fracnt + |\im\la|$, $\mce(\la)\rho^b$ is bounded from $\ltx$ to $L^2(\p X)$. Moreover, for $\la$ in any compact set of $\Omega_0$, there exists $C > 0$ such that 
\beq
\|\mce(\la)\rho^b f\|_{L^2(\p X)} \leq C\|f\|_{\ltx}.
\eeq
\end{prop}

\subsection{Asymptotics of the radiation fields. }
Let $\mcv(X)$ be the Lie algebra of smooth vector fields on $X$, and $\mcv_e(X) = \rho \mcv(X)$ be the Lie algebra of edge vector fields, see Mazzeo \cite{Ma2}. These are also called $0$-vector fields, see \cite{MM, M, MSV}. We need the definition of the weighted edge Sobolev space
\beq
\rho^b H_e^k(X) \doteq \{u = \rho^b v:  v \in \ltx,  V_1 V_2\cdots V_k v \in \ltx, \ \  V_i \in \mcv_e(X), 0\leq i \leq k\}.
\eeq

\begin{lemma}\label{eisenreg}
Assume $\{\la\in \mc : |\re\la | > C_2, |\im\la| < C_1\}$ as in Theorem \ref{main1} and $b > \fracnt + |\im\la|$. For $m \in \mn$ and $f\in \rho^bH_e^\infty(X)$, there exists a constant $C$ depending on $m$ and $f$ such that
\beq
\|\mce(\la) f\|_{L^2(\p X)} \leq C |\la|^{-m}.
\eeq
\end{lemma}

\bpf
Assume that $f = \rho^b v$. From the identity $R(\la)(\lap_g - \frac{n^2}{4} - \la^2) = \Id$, we get 
\beq
\la^2 R(\la) = R(\la)\lap_g - \frac{n^2}{4}R(\la) - \Id.
\eeq
Apply this to definition \eqref{eisen} of $\mce(\la)$, we have
\begin{equation}\label{eisenest2}
\begin{split}
\la^2 \mce(\la) \rho'^b v &= \lim_{\rho\rightarrow 0} \rho^{-i\la -\fracnt} R(\la)\lap_g \rho'^b v - \lim_{\rho\rightarrow 0} \rho^{-i\la -\fracnt} \frac{n^2}{4}R(\la)\rho'^b v - \lim_{\rho\rightarrow 0} \rho^{-i\la -\fracnt + b} v\\
& = \mce(\la)\rho'^b(\rho'^{-b}\lap_g \rho'^b v - \frac{n^2}{4} v). 
\end{split}
\end{equation}
The last term in the first line vanishes because $b > \fracnt + |\im\la|$.  Repeating this argument, we get for $k\in \mn$ that
\beq
\la^{2k}\mce(\la)\rho^b v = \mce(\la)\rho^b(\rho^{-b}\lap_g \rho^b  - \frac{n^2}{4})^k v.
\eeq
Notice that by our choice of the boundary defining function, $\rho = x$ when \eqref{prod} is valid. Using the expression of $\lap_g$ \eqref{eq4.1.1}, we see that
\beq
x^{-b} \lap_g x^b = -(x\frac{\p}{\p x} + b)^2 + n(x\frac{\p}{\p x} + b) -x\gamma (x\frac{\p}{\p x} + b) + x^2\lap_H.
\eeq
So $x^{-b} \lap_g x^b$ is an edge differential operator of order $2$ in the sense that it can be written as a finite sum of $V_i V_j$ with $V_i, V_j\in \mcv_e(X)$. Since $v\in H^\infty_e(X)$, by Proposition \ref{higheis} and \ref{loweis}, there exists a constant $C$ such that
\begin{gather*}
\|\mce(\la)\rho^b v\|_{L^2(\p X)} \leq C |\la|^{\fnt + \kappa - 1 - k}.
\end{gather*}
\epf

The definition of radiation fields is recalled in the introduction. In \cite{Sa}, S\'a Barreto proved  that the radiation field is related to the transposed Eisenstein function by 
\begin{equation}\label{radeis}
\int_\mr e^{-i\la s}\mcr_+(f_1, f_2)(s, y)ds = i\la \mce(\la)(f_2 + \la f_1). 
\end{equation}
This is valid for $\im\la < -\fracnt$ if $f_1, f_2 \in C^\infty_0(X)$. By Proposition \ref{higheis} and \ref{loweis}, the formula \eqref{radeis} remains valid for $\la$ in $\Omega_0 = \{\la\in\mc\backslash 0: |\im\la| < \delta_{0}\}$ with $f_1, f_2 \in \rho^b \ltx$ and $b > \fracnt + |\im\la|$. The right hand side of (\ref{radeis}) is holomorphic in $\Omega_0$, so is the left hand side. 

As mentioned in the introduction see also \cite{Sa}, the spectrum of $\lap_g$ consists of absolutely continuous spectrum $[\fnf, \infty)$ and finitely many eigenvalues in $(0, \fnf)$. So we can decompose $L^2(X) = L^2_{ac}(X)\bigoplus L^2_{pp}(X)$ according to the spectrum, where $L^2_{pp}(X)$ is spanned by the eigenfunctions and $L^2_{ac}(X)$ is the orthogonal complement. As a consequence of  Theorem 5.1 of \cite{Sa}, we have 
\beq
\|\mcr_\pm (0, f_2)\|_{L^2(\mbr_s\times \p X)} = \|f_2\|_{L^2(X)},  \ \ \forall f_2 \in L^2_{ac}(X).
\eeq

Finally we prove the main result about radiation fields.  
\begin{theorem}\label{main2}
Assume that $0 \leq \epsilon < \delta_{0}$ and $b > \fracnt + \epsilon$. For $f_1 \in \rho^bH_e^\infty(X)$ and $f_2 \in \rho^bH_e^\infty(X) \cap L_{ac}^2(X)$, the radiation field satisfies
\beq
\mcr_{+}(f_1, f_2)(s, y) = e^{-\epsilon |s|} \mcr_0(s, y), 
\eeq
where $\mcr_0(s, y) \in L^\infty(\mbr; L^2(\p X))$ and smooth in $s \in \mbr\backslash\{0\}$. 
\end{theorem}
\begin{proof}
We will prove the theorem for $f_1 = 0$. The case when $f_1\neq0$ is similar. First, we take inverse Fourier transform in $\la$ of \eqref{radeis} to get  
\begin{equation}\label{eqrad1}
\mcr_+(0, f_2)(s, y) = \frac{1}{2\pi} \int_\mbr e^{is\la}\la \mce(\la)f_2 d\la.
\end{equation} 
By assumption, $\la \mce(\la)$ is holomorphic in $\Omega_0$ with only a possible pole at $\la = 0$. The order of the pole of $\mce(\la)$ at $\la = 0$ follows from that of $R(\la)$.  Here we use an argument in the Euclidean potential scattering, see  Section 2.10 of Melrose \cite{M}. Since $R(\la)$ is meromorphic  near $\la = 0$, we can write
\begin{equation}\label{eq0expan}
R(\la) = \sum_{j = 1}^{N_0} \frac{R_{-j}}{\la^j} + R_0(\la),
\end{equation}
where $R_0(\la)$ is holomorphic, and $R_{-j}$ are the residue operators with finite rank. From the spectral theorem, $R(\la)$ is bounded on $\ltx$ for $\im\la < 0, \la^2 \notin \text{spec}(\lap_g - \frac{n^2}{4})$, and
\begin{equation*}
\|R(\la)\|_{\ltx\rightarrow\ltx} = \frac{1}{\text{dist}(\la^2, \text{spec}(\lap_g - \frac{n^2}{4}))}.
\end{equation*}
Therefore, for $\im\la < 0$ near $0$, we have $\|R(\la)\| \sim 1/|\la|^2$. This means that in the expansion (\ref{eq0expan}), there are at most $2$ nontrivial singular terms $R_{-2}\la^{-2}$ and $R_{-1}\la^{-1}$. Hence $R(\la)$, as well as $\mce(\la)$, has a pole of order at most 2 at $\la = 0$.

\begin{figure}
\scalebox{0.9} 
{
\begin{pspicture}(0,-2.0)(10.68,2.0)
\rput(5.02,-1.0){\psaxes[linewidth=0.04,labels=none,ticks=none,ticksize=0.10583333cm]{->}(0,0)(-5,-1)(5,3)}
\psline[linewidth=0.04cm,arrowsize=0.05291667cm 2.3,arrowlength=2.0,arrowinset=0.1]{<-}(1.9,-0.52)(1.9,0.2)
\psline[linewidth=0.04cm](6.46,0.2)(8.38,0.2)
\psline[linewidth=0.04cm,arrowsize=0.05291667cm 2.3,arrowlength=2.0,arrowinset=0.1]{->}(8.38,-1.0)(8.38,-0.28)
\psline[linewidth=0.04cm,arrowsize=0.05291667cm 2.3,arrowlength=2.0,arrowinset=0.12]{->}(1.9,0.2)(6.7,0.2)
\psline[linewidth=0.04cm](8.38,0.2)(8.38,-0.28)
\psline[linewidth=0.04cm](1.9,-1.0)(1.9,-0.52)
\usefont{T1}{ptm}{m}{n}
\rput(5.32,0.5){$i\epsilon$}
\usefont{T1}{ptm}{m}{n}
\rput(5.27,-1.3){$o$}
\usefont{T1}{ptm}{m}{n}
\rput(2.01,-1.375){$-K$}
\usefont{T1}{ptm}{m}{n}
\rput(8.43,-1.375){$K$}
\usefont{T1}{ptm}{m}{n}
\rput(9,-0.415){$\Gamma_K$}
\usefont{T1}{ptm}{m}{n}
\rput(1.22,-0.415){$\Gamma_{-K}$}
\usefont{T1}{ptm}{m}{n}
\rput(5.7,1.915){$\im\la$}
\usefont{T1}{ptm}{m}{n}
\rput(10.59,-1){$\re\la$}
\end{pspicture} }
\caption{Contour deformation in the proof of Theorem \ref{main2}.}
\label{fig8}
\end{figure}
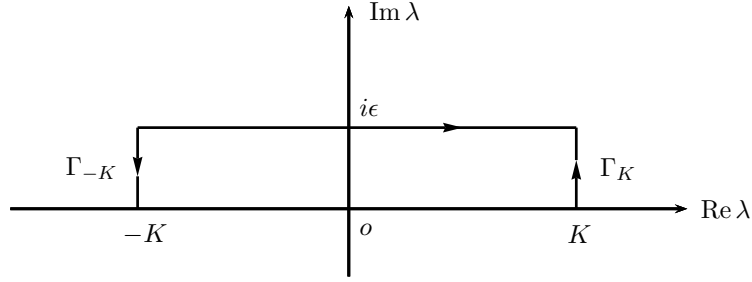

By taking one derivative in $s$ of \eqref{eqrad1}, we get 
\beq
\p_s\mcr_+(0, f_2)(s, y) = \frac{i}{2\pi} \int_\mbr e^{is\la}\la^2 \mce(\la)f_2 d\la,
\eeq
and the integrand now is holomorphic in the strip $|\im\la| < \delta_0$. We can change the contour of integration as in Figure \ref{fig8}. By Lemma \ref{eisenreg},  
\begin{equation*}
\begin{split}
\|\int_{\Gamma_K} e^{i\la s} \la^2\mce(\la)f_2 d\la\|_{L^2(\p X)} & = \|\int_0^\epsilon e^{iKs - ts} (K+ it)^2\mce(K + it)f_2 dt\|_{L^2(\p X)}\\
& \leq C e^{-\epsilon s} K^{-1} \rightarrow 0 \text{ as } K \rightarrow \infty.
\end{split}
\end{equation*}
The integral along $\Gamma_{-K}$ vanishes similarly. By Cauchy's theorem, 
\beq
\p_s\mcr_+(0, f_2)(s, y) = e^{-\epsilon s} \mcr_{0, +}(s, y), \text{ where } \mcr_{0, +}(s, y) = \frac{i}{2 \pi} \int_{-\infty}^\infty e^{i\la s}  (\la + i\epsilon)^2\mce(\la + i\epsilon)f_2 d\la.
\eeq
An application of Lemma \ref{eisenreg} shows that $\mcr_{0,+}(s, y) \in L^\infty(\mbr_s; L^2(\p X))$ and smooth in $s$. Integrating from $0$ to $s$, we get $\mcr_+(0, f_2)(s, y) = \Psi(y) + e^{-\epsilon s} \Psi_{0, +}(s, y)$ where $\Psi_{0, +}(s, y) \in L^\infty(\mbr_s; L^2(\p X))$ and smooth in $s$. However, we know that $\mcr_+(0, f_2) \in L^2(\mbr_s\times \p X)$ if $f_2\in L^2_{ac}(X)$. It is easy to see that $e^{-\epsilon s} \Psi_{0, +}(s, y) \in L^2([0, \infty)_s\times \p X)$. Therefore, the function $\Psi(y)$ has to vanish almost everywhere. So we have
\beq
\mcr_+(0, f_2)(s, y) = e^{-\epsilon s} \Psi_{0, +}(s, y).
\eeq 
On the other hand, we can deform the contour to the lower half plane and use the same argument to get $\mcr_+(0, f_2)(s, y) = e^{\epsilon s} \Psi_{0, -}(s, y)$. To finish the proof, we define $\mcr_0$ to be $\Psi_{0, +}$ for $s \geq 0$ and $\Psi_{0, -}$ for $s < 0$. 
\end{proof}

\appendix
\section{Positivity of the injectivity radius}
For any point $z$ on a general complete Riemannian manifold $(X, g)$, the injectivity radius $\inj(z)$ is defined as the supremum of $\delta$ such that the exponential map $\exp_z: B_\delta(0) \subset T_z X \rightarrow X$ is a diffeomorphism. The injectivity radius of $X$ is defined as $\inj(X) = \inf_{z \in X}\inj(z)$. It is well-known that for compact manifold, $\inj(X) > 0$, see e.g. 3.80 of \cite{GFL}.

\begin{lemma}\label{inject}
For $(\intx, g)$ asymptotically hyperbolic, the injectivity radius $\inj(\intx) > 0$.
\end{lemma}
\begin{proof}
Let $\rho$ be a boundary defining function of $\p X$. In the following, we denote 
\beq
U_\epsilon = \{z \in \intx: \rho(z) < \epsilon\} \text{ and } U_\epsilon^c = \{z \in \intx: \rho(z) \geq \epsilon \}.
\eeq
Since $(\intx, g)$ is asymptotically hyperbolic, we can assume that the sectional curvature is less than $-\kappa_0 < 0$ in $U_{\epsilon_0}$, for some $\epsilon_0$ sufficiently small. 

We first prove that $\forall z \in U_{2\epsilon_0/3}$, $\inj(z) \geq \delta_1$  for some $\delta_1 > 0$. Since $\overline{U_{2\epsilon_0/3}}$ is closed and $U^c_\epsilon$ is compact, there exists some $\delta_1 > 0$ such that $\text{dist} (\overline{U_{2\epsilon_0/3}}, U^c_{\epsilon_0}) \geq \delta_1$.  Because the sectional curvature is negative in $U_{\epsilon_0}$, by the conjugate point comparison theorem Corollary 11.3 of \cite{L}, there cannot be more than one geodesics between any two points $z, z' \in U_{2\epsilon_0/3}$ that lies inside $U_{\epsilon_0}$. Therefore, for any vector $V \in T_z \intx$ with $\|V\| \leq \delta_1$, the exponential map is a  diffeomorphism. 

Now consider $U^c_{\epsilon_0/3} \subset X$ as a compact manifold. Then $\inj(U^c_{\epsilon_0/3}) \geq \delta_2$ for some $\delta_2 > 0$. Since $U^c_{\epsilon_0/3}$ and $U_{2\epsilon_0/3}$ cover $\intx$, we conclude that $\inj(\intx) \geq \min(\delta_1, \delta_2) > 0$. 
\end{proof}
%



\begin{thebibliography}{99}
\bibitem{BVW} Baskin, Dean; Andr\'as Vasy; Jared Wunsch {\em Asymptotics of radiation fields in asymptotically Minkowski space.}  arXiv:1212.5141 (2012).


\bibitem{D}  Duistermaat, Johannes {\em Fourier integral operators.} Vol. 130. Springer, 1996.


\bibitem{F} Friedlander,  Gerard {\em Radiation fields and hyperbolic scattering theory.} In Mathematical Proceedings of the Cambridge Philosophical Society, vol. 88, no. 03, pp. 483-515. Cambridge University Press, 1980.

\bibitem{F1}  Friedlander, Gerard. {\em Notes on the wave equation on asymptotically Euclidean manifolds.} Journal of Functional Analysis 184, no. 1 (2001): 1-18.

\bibitem{GFL}  Gallot, Sylvestre; Hulin, Dominique; Lafontaine, Jacques {\em Riemannian geometry.} Vol. 3. Berlin: Springer-Verlag, 1990.

\bibitem{Gr} Graham, Robin {\em Volume and area renormalizations for conformally compact Einstein metrics.} Rend. Circ. Mat. Palermo (2) Suppl. No. 63 (2000), 31-42. 

\bibitem{G} Guillarmou, Colin {\em Absence of resonance near the critical line on asymptotically hyperbolic spaces.} Asymptot. Anal. 42 (2005), no. 1-2, 105-121.

\bibitem{G1} Guillarmou, Colin {\em Meromorphic properties of the resolvent on asymptotically hyperbolic manifolds.} Duke Math. J. 129 (2005), no. 1, 1-37.






\bibitem{JS1} Joshi, Mark; S\'a Barreto, Ant\^onio {\em Inverse scattering on asymptotically hyperbolic manifolds.} Acta Math. 184 (2000), no. 1, 41-86.




\bibitem{LP} Lax, Peter; Ralph, Phillips. {\em Scattering theory.} Revised second edition. Vol. 26. Academic press, 1990.
\bibitem{L} Lee, John {\em Riemannian manifolds. An introduction to curvature.} Graduate Texts in Mathematics, 176. Springer-Verlag, New York, 1997.

\bibitem{Ma1} Mazzeo, Rafe {\em The Hodge cohomology of a conformally compact metric.} J. Differential Geom. 28 (1988), no. 2, 309-339.

\bibitem{Ma2} Mazzeo, Rafe {\em Unique continuation at infinity and embedded eigenvalues for asymptotically hyperbolic manifolds.} Amer. J. Math. 113 (1991), no. 1, 25-45.

\bibitem{Ma3} Mazzeo, Rafe. {\em Elliptic theory of differential edge operators I.} Communications in Partial Differential Equations 16, no. 10 (1991): 1615-1664.

\bibitem{MM} Mazzeo, Rafe; Melrose, Richard {\em Meromorphic extension of the resolvent on complete spaces with asymptotically constant negative curvature.} J. Funct. Anal. 75 (1987), no. 2, 260-310.

\bibitem{M} Melrose, Richard {\em Geometric Scattering Theory}. Stanford Lectures, Cambridge University Press, 1995. 


\bibitem{MSV} Melrose, Richard; S\'a Barreto, Ant\^onio; Vasy, Andr\'as {\em Analytic continuation and semiclassical resolvent estimates on asymptotically hyperbolic spaces.} Communications in Partial Differential Equations 39, no. 3 (2014): 452-511.


\bibitem{Sa} S\'a Barreto, Ant\^onio {\em Radiation fields, scattering, and inverse scattering on asymptotically hyperbolic manifolds.} Duke Math. J. 129 (2005), no. 3, 407-480.

\bibitem{SaW}  S\'a Barreto, Ant\^onio; Wunsch, Jared {\em The radiation field is a Fourier integral operator.} Ann. Inst. Fourier (Grenoble) 55 (2005), no. 1, 213-227.

\bibitem{Pre} S\'a Barreto, Ant\^onio; Wang, Yiran {\em The scattering relation on asymptotically hyperbolic manifolds.} Preprint in arXiv, 2014.

\bibitem{So} Sogge, Christopher {\em Fourier integrals in classical analysis.} Vol. 105. Cambridge University Press, 1993.

\bibitem{V} Vasy, Andr\'as  {\em Microlocal analysis of asymptotically hyperbolic spaces and high-energy resolvent estimates.} Inverse Problems and Applications: Inside Out II 60 (2013): 487.

\end{thebibliography}
\end{document}